\documentclass[a4paper,reqno]{amsart}
\usepackage{amsmath,amssymb,amsfonts,amsthm}
\usepackage{mathtools}

\usepackage{fancyhdr,fancyvrb}

\usepackage{graphicx}

\usepackage{nextpage}

\usepackage[dvipsnames,usenames]{color}

\usepackage[colorlinks=true, linkcolor=blue, citecolor=red]{hyperref}

\usepackage{tikz}
\usepackage{tikz-3dplot}
\usepackage{mathrsfs}
\usepackage{stmaryrd}

\makeatletter
\@namedef{subjclassname@2020}{%
	\textup{2020} Mathematics Subject Classification}
\makeatother

\newcommand{\N}{\mathbb{N}}

\newcommand{\bv}{\mathrm{BV}}
\newcommand{\bx}{\mathbf{x}}

\newcommand{\bfx}{\textbf{x}}

\newcommand{\cA}{\mathcal A}

\newcommand{\cC}{\mathcal C}

\newcommand{\cE}{\mathcal E}
\newcommand{\cF}{\mathcal F}
\newcommand{\cG}{\mathcal G}
\newcommand{\cH}{\mathcal H}

\newcommand{\cL}{\mathcal L}

\newcommand{\cO}{\mathcal O}

\newcommand{\de}{\mathrm{d}}

\newcommand{\loc}{\mathrm{loc}}

\newcommand{\rr}{\mathrm{r}}

\newcommand{\eps}{\varepsilon}

\newcommand{\abs}[1]{\left|#1\right|}

\newcommand{\dd}{\,\mathrm{d}}
\newcommand{\dho}{\,\dd{\mathcal H}^{1}}

\newcommand{\dist}[2]{\operatorname{dist}(#1,#2)}

\newcommand{\supp}{\operatorname{supp}}

\renewcommand{\O}{\Omega}

\newcommand{\R}{\mathbb{R}}
\usepackage{soul}

\newcommand{\ho}{{\mathcal H}^{1}}

\renewcommand{\geq}{\geqslant}
\renewcommand{\leq}{\leqslant}
\newcommand{\wto}{\rightharpoonup}
\newcommand{\wsto}{\stackrel{*}{\rightharpoonup}}

\newcommand{\average}{{\mathchoice {\kern1ex\vcenter{\hrule
				height.4pt width 8pt depth0pt}
			\kern-11pt} {\kern1ex\vcenter{\hrule height.4pt width 4.3pt
				depth0pt} \kern-7pt} {} {} }}

\mathchardef\emptyset="001F

\newtheorem{theorem}{Theorem}
\newtheorem{definition}{Definition}
\newtheorem{lemma}{Lemma}
\newtheorem{proposition}{Proposition}
\newtheorem{remark}{Remark}

\title{Phase-field approximation of an epitaxial growth model with adatoms}

\author[G.~Fissore]{Gabriele Fissore}
\address[G.~Fissore]{Department of Mathematics - IMAPP, Radboud University, Nijmegen, The Netherlands}
\email{gabriele.fissore@ru.nl}

\begin{document}
	
	\begin{abstract}
We perform a mass constrained phase-field approximation for a model that describes the epitaxial growth of a two-dimensional thin film on a substrate in the context of linearised elasticity. The approximated model encodes a variable on the free surface of the film, that physically is interpreted as an adatom density.
	\end{abstract}
	
	\maketitle

\section{Introduction}
\allowdisplaybreaks
This paper is the continuation of the study contained in \cite{CriFis}. In that paper, we investigated the mathematical formulation of a model on epitaxial growth, namely the creation of a thin film onto a fixed crystalline substrate.  If the atomic structure of the film and the substrate slightly differs, a stress inside the deposited film is created. As a consequence, flat configurations of the film are not energetically optimal, and the creation of a corrugated surfaces is preferred together with defects like fractures (see \cite{Grin}). This process is relevant in industry as it is used in many applications such as the production of micro-components, microchips, coatings and graphene. The physical literature on crystal growth is extremely vast, we limit to mention the work \cite{SpeTer97} by Spencer and Tersoff.\\

In mathematics, this is an interesting topic to deal with, since it is possible to describe equilibrium configurations in a variational setting. We refer to the work  \cite{FriGur} by Fried and Gurtin for an exhaustive description of epitaxial growth. The problem can be rephrased in terms of linearised elasticity, if we assume that the thin film behaves like an elastic material. Since in the mathematical model a surface energy is taken into account,  the interplay between it and the elastic energy leads to configurations of the deposited crystal that are of mathematical interest. More specifically, an equilibrium configuration can exhibit pattern formations and defects such as fractures within the material. Those configurations, in the process of minimisation of an energy, are preferred since smooth surfaces might create a tension inside the crystalline lattice of the film.\\
This analysis has been carried out in the paper \cite{FonFusLeoMor07}, by Fonseca, Fusco, Leoni and Morini, in which the epitaxial growth is modelled in the two-dimensional case, where the film is constrained to be a graph of a function. Here, the surface energy considered corresponds to the length of the graph, up to two constants that, in the process of minimization, tell us whether it is convenient to have the substrate exposed or not. We remark that this mechanism is at the base of the \textit{island formation}. About this topic, we refer to the numerical simulations shown in \cite{BonCha02} by Bonnetier and Chambolle. In that framework, bounded variation functions are used to describe stress free configurations. The three-dimensional case has been handled in \cite{CriFriSol} by Crismale, Friedrich and Solombrino, in the context of $GSBD$ functions, a functional space introduced by Dal Maso in \cite{DalMasoGSBD}. In both models the fractures of the film are vertical inside the crystal. Again in \cite{BonCha02}, a phase-field formulation of the problem is introduced.\\
The work carried out in \cite{CriFis}, by Cristoferi and the author of this paper, adds a new result in this direction. We developed a model that introduces a new variable to the problem, based on an idea introduced in \cite{CarCriDie}. More specifically, we considered the case in which a density of atoms lies on the surface of the crystal. Unlike those that are part of the film, these atoms are free to move along the surface. They are referred to as \textit{adatoms}, and their relevance in industrial applications is highlighted by the fact that pattern formation is influenced by their presence. Many of the previous papers mentioned neglect such an additional variable.\\
The main result in \cite{CriFis} is a relaxation theorem with mass and density constraints. We modelled fractures in a material using new techniques, and it is shown that previous strategies for modelling material fractures do not apply when an adatom density is present on the surface. \\
Now, briefly explain the relevance and the main ideas behind the techniques introduced in \cite{CriFis}, many of which will be used again in the present paper. For further details in what follows, we refer to Section \ref{mathseting}.\\
First of all, let $\O_h$ be the sub-graph of a Lipschitz function $h:(a,b)\to \R{}$, namely
\begin{align*}
	\O_h\coloneqq\{\bx=(x,y)\in\R^2\, |\, y<h(x)\},
\end{align*}
 which represents the region occupied by the thin film. Inside the bulk, we condider $v\in H^1(\O;\R^2)$ as the displacement variable. The mentioned adatom density is seen as a Radon measure $\mu=u\ho\llcorner\Gamma_h$, where $\Gamma_h=\partial\O_h\cap(a,b)$ is the interface of the film and $u\in L^1(\Gamma_h)$. We study a functional of the form
\begin{align}\label{unrelaxedintro}
	\cH(\O_h,v,\mu)\coloneqq
	\displaystyle\int_{\O_h}W\big(E(v)-E_0(y)\big)\ \de \bfx+\int_{\Gamma_h}\psi(u)\dho.
\end{align}
The first term of \eqref{unrelaxedintro} is the stored elastic energy. In the linearised setting, it is given by $W:\R^{2\times 2}\to \R{}$ and the elasticity tensor $\mathbb{C}$, and it is defined as
\begin{align*}
	W(A)\coloneqq\frac{1}{2} A\cdot \mathbb{C}  [A].
\end{align*}
Moreover, 
\begin{align*}
	E(v)\coloneqq\frac{1}{2}(\nabla v + \nabla^\top v)
\end{align*}
is the symmetric gradient and $E_0:\R\to \R^{2\times2}$ represents the starting configuration of the substrate and the film, in which their lattices are perfectly aligned. \\
The surface term in \eqref{unrelaxedintro} penalises not only the length of the graph, but also an elevate adatom density, given the presence of a Borel function  $\psi:\R\to(0,\infty)$. \\
The main difficulty is the relaxation of \eqref{unrelaxedintro} for the surface term. In order to see vertical fractures inside the film, we use bounded variation functions and a suitable topology in the relaxation. Such a topology (see Definition \ref{topologyoldpaper}) is made in such a way that fractures are limiting configurations of sequences of regular ones. One can expect on the interface a similar result as in \cite{CarCriDie}, in which the surface energy has the form
\begin{align*}
	\int_{\Gamma_h}\widetilde{\psi}(u)\dho,
\end{align*} where $\psi$ is the convex sub-additive envelope of $\psi$ (see Definition \ref{psitilde}). \\
However, we proved that on vertical fractures we have a different behaviour. Indeed, let $\Gamma^c_h$ be the cut part of $h$, namely
\begin{align*}
	\Gamma^c_h\coloneqq\{\bfx\in\R^2 \,|\, x\in(a,b),\, h(x)\leq y<h^-(x)\},
\end{align*}
where $h^-(x)=\liminf_{y\to x} h(y)$ (see Definition \ref{cuts}). We have that the interface energy on fractures is relaxed to a term of the form
\begin{align*}
	\int_{\Gamma_h^c}\psi^c(u)\ \de \ho.
\end{align*} 
Here, $\psi^c$ (see Definition \ref{psic}) is defined  as
\begin{align*}
		\psi^c(s)\coloneqq\min\{\widetilde{\psi}(r)+\widetilde{\psi}(t) \,|\, s=r+t\}.
\end{align*}
Heuristically, since we are approximating a fracture with regular functions, we expect that a cut is the limit of the contributions given by functions on the left and on the right of it. In that sense $\psi^c$ detects the best possible way to do so. \\
We proved that $\cG$ is the relaxation of $\cH$, where
\begin{align}\label{introductionG}
\cG(\O_h,v,\mu)\coloneqq\int_{\O_h}W\big(E(v)-E_0(y)\big)\ \de \textbf{x}+\int_{\widetilde{\Gamma}_h}\widetilde{\psi}(u)\dho+\int_{\Gamma^c_h}\psi^c(u)\dho+\theta\mu^s(\R^2).
\end{align}
Here, $\theta>0$ is the recession coefficient (see Definition \ref{theta}), $\mu^s$ is the singular part of $\mu$ with respect to $\ho\llcorner\Gamma_h$ and $\widetilde{\Gamma}_h\coloneqq\Gamma_h\setminus \Gamma_h^c$ is the regular part of the graph of $h$.

In this paper, we develop a phase-field model for the sharp functional \eqref{introductionG}. The reason behind this analysis is motivated by the fact that sharp models are, in general, hard to tackle numerically. Therefore, we would like to provide with an approximation of \eqref{introductionG} by only using variables which are functions. Consider
\begin{align*}
	Q\coloneqq(a,b)\times\R{}\quad\text{and}\quad
	Q^+\coloneqq Q\cap \{y>0\}.
\end{align*}
We aim to approximate the set variable $\O$ to a phase-field sequence $(w_\eps)_\eps\subset H^1(Q^+)$ such that, for every $\eps>0$, we have $0\leq w_\eps\leq1$ and $w_\eps\to \chi_\O$ in $L^1(Q^+)$, as $\eps\to0$. In the same spirit, the Radon measure $\mu=u\ho\llcorner\Gamma_h$ is seen as the limit (in the sense of the $\ast$-weak convergence) of a sequence $(u_\eps)_\eps\subset L^1(Q)$.\\
To this end, we introduce the functional, for $w\in H^1(Q^+)$, $v\in H^1(Q;\R^2)$ and $u\in L^1(Q)$ 
\begin{align}\label{introductiongeps}
	\nonumber	\cG_\eps(w,v,u)\coloneqq\int_{Q^+}(w(\bfx)&+o(\eps)) W\big(E(v)-E_0(y)\big)\ \de\textbf{x}\\[5pt]
		&+	\frac{1}{\sigma}\int_{Q^+} \Big[\eps|\nabla w|^2+\frac{1}{\eps}P\big(w\big)\Big]\psi(u)\ \de\textbf{x},
\end{align}
where $P$ is a double-well potential and  $\sigma\coloneqq2\int_0^1 \sqrt{P(t)}\ \de t$. We refer to  Subsection \ref{phasefieldsubsection} for a precise formulation. The main result of this work is proving that $\cG_\eps\stackrel{\Gamma}{\to}\cG$, as $\eps\to0$, with respect to a suitable topology (see Definition \ref{topology} and Theorems \ref{mainresult1} and \ref{mainresult2}). \\
We remark that the role of $o(\eps)$ in \eqref{introductiongeps} and also in the functionals presented in this section is to ensure the compactness of the minimising sequences. \\
The proof makes use of the strategies introduced in \cite{BonCha02} and \cite{CarCri}. In order to merge them, we make large use of the new techniques introduced in \cite{CriFis}, as they proved to be solid enough to be adapted in a phase-field formulation.\\

There is a rich literature about phase-field models who aimed to the same outreach. We briefly recall the main contributions that are related to the present paper. In what follows, $w$ plays the role of the phase-field variable.\\
In the seminal paper on the topic \cite{AmbTor} the authors introduced the Ambrosio-Tortorelli functional. For a bounded domain $\O\subset \R^n$ and a given $g\in L^\infty(\O)$, $v,w\in H^1(\Omega)$, which satisfy $(v,w)=(g,1)$ on $\partial\O$, we define 
\begin{align}\label{atfunctional}
	AT_\eps(v,w)\coloneqq\int_\O (o(\eps)+w^2)\abs{\nabla v}^2\ \de\bfx+\frac{1}{2}\int_\O\Big(\eps\abs{\nabla w}^2+\frac{1}{4\eps}(w-1)^2\Big)\ \de\bfx.
\end{align}
We remark that  the definition of $AT_\eps$ does not coincide with the original one given by Ambrosio and Tortorelli. However, the above modern formulation of  $AT_\eps$ is made in such a way that is easier to be dealt numerically for applications. \\
We have that $AT_\eps$ is the phase-field formulation for the Mumford-Shah functional, in the sense that it $\Gamma$-converges, with respect to a suitable topology, to
\begin{align*}
	MS(v)=\int_\O \abs{\nabla v}^2\ \de\bfx+\cH^{n-1}(\partial\O\cap\{v\neq g\}) +\cH^{n-1}(S_v),
\end{align*}
where $v\in SBV^2(\O)$ and $S_v$ is the jump set of $v$. The Mumford-Shah functional is related to the image segmentation problem. For more details we refer to \cite{MumSha}. \\

Closer to our topic, phase-field models about epitaxial growth are present in \cite{BonCha02} and \cite{FonFusLeoMor07}. In particular, in both papers, but using independent techniques, the authors studied a phase-field formulation for the functional \eqref{unrelaxedintro} in case $\psi\equiv 1$. Their models include two constants $\sigma_c,\sigma_s>0$ that allow to describe the wetting or non-wetting regimes. This translates into the fact that the free surface of the film is penalised depending on its location. \\ 
In the same setting as in \eqref{unrelaxedintro}, the relaxed energy is given by
	\begin{align*}
		BC(\O_h,v)\coloneqq
		\displaystyle\int_{\O_h}W\big(E(v)-E_0(y)\big)\ \de \textbf{x}&+(\sigma_s\wedge\sigma_c)\ho(\Gamma_h\cap\{y=0\})\\[5pt]
		&+\sigma_c \Big(\ho(\Gamma_h\cap\{y>0\})+2\sum_{\bfx\in \Gamma_h^c}(h^-(x)-h(x))\Big),
	\end{align*}
	in which the contribution given by a fracture is counted twice.
In this context, in \cite{BonCha02}, they defined
\begin{align}\label{bcfunctional}
\nonumber BC_\eps(v,w)\coloneqq\int_{Q^+}(w&+o(\eps)) W\big(E(v)-E_0(y)\big)\ \de\textbf{x}\\[5pt]
&+	2\sigma_c\Big(\int_{Q^+}\frac{4\eps}{\pi^2} |\nabla w|^2+\frac{1}{\eps}w(1-w)\ \de\textbf{x}\Big),
\end{align}
and proved that $BC_\eps$ $\Gamma$-converges to $BC$ by using a topology similar to the one described in the present contribution.\\

The last model for a phase-field, and in that case by taking into account an adatom density is given by Caroccia and Cristoferi in \cite{CarCri}. In this paper the authors proved a general $\Gamma$-convergence result which is immediately applied to epitaxial growth in a domain in $\R^n$. They defined the energy
\begin{align*}
	CC_\eps(w,u)\coloneqq\frac{1}{\sigma} \int_{\O} \Big(\eps \abs{\nabla w}^2+\frac{1}{\eps}P(w)\Big)\psi(u)\ \de\bfx.
\end{align*}
They proved that, $CC_\eps$ $\Gamma$-converges to
\begin{align*}
	CC(E,\mu)\coloneqq \int_{\partial^\ast E}\widetilde{\psi}(u)\ \de\cH^{n-1}+\theta\mu^s(\R^n),
\end{align*}
in which $\mu=u\cH^{n-1}\llcorner\partial^\ast E+\mu^s$, where $\partial^\ast E$ denotes the reduced boundary of the set (see \cite{Mag}).\\
 We also mention the work by Conti, Focardi and Iurlano \cite{ConFocIur} and Focardi \cite{Foc} on the topic. \\

We notice that all the above phase-field models presented in this introduction displays two phases. In $BC_\eps$, $CC_\eps$ (and in $\cG_\eps$ for the present paper), those two phases are entirely encoded by the two wells of the potential. In this sense, it might seem that $AT_\eps$ has only one phase as the potential has one well at $1$. However, the second phase of $AT_\eps$ is $0$, since we expect that the phase-field variable $w$ vanishes on the jump set of $v$. This difference us further remarked by the fact that in $BC_\eps$, $CC_\eps$ and  $\cG_\eps$ the set $\O$ plays the role of a variable.

\section{Preliminaries}\label{preliminaries}
In this section we explain the main ideas of $\Gamma$-convergence and the technical tools needed to formulate our problem.
\subsection{$\Gamma$-convergence}
We introduce the definition of convergence of functional used in this paper. We refer to \cite{DM93} for a complete description and to the proofs of the following results. We start with the topological definition of $\Gamma$-limit. Let $X$ be a topological space and denote by  $\cO(x)$, for $x\in X$, the set of all open neighbourhoods of $x$. As a notation,  we write $\overline{\R}=\R\cup\{\pm\infty\}$.
\begin{definition}\label{gammaconvergenza}
	The \emph{$\Gamma$-lower limit} and the \emph{$\Gamma$-upper limit} of a sequence of functionals $F_\eps:X\to\overline{\R}$ are defined as
	\begin{align*}
		(\Gamma{-}\liminf_{\eps\to0} F_\eps) (x)&\coloneqq \sup_{U\in \cO(x)}\liminf_{\eps\to0} \inf_{y\in U}F_h(y),\\
		(\Gamma{-}\limsup_{\eps\to0} F_\eps) (x)&\coloneqq \sup_{U\in \cO(x)}\limsup_{\eps\to0} \inf_{y\in U}F_\eps(y).
	\end{align*}
	If there exists $F:X\to\overline{\R}$ such that
	\begin{align*}
		F=\Gamma{-}\liminf_{\eps\to0} F_\eps=\Gamma{-}\limsup_{\eps\to0} F_\eps,
	\end{align*}
	we write
	\begin{align*}
		F=\Gamma{-}\lim_{\eps\to0} F_\eps
	\end{align*}
	and we say that $F_\eps\stackrel{\Gamma}{\to}F$ in $X$ as $\eps\to0$.
\end{definition}
From Definition \ref{gammaconvergenza} is clear that the notion of $\Gamma$-convergence depends on the topology taken into account. If we make some additional assumptions on the topological space $X$, we can give a sequential characterization of the $\Gamma$-convergence. In case $X$ satisfies the first axiom of countability, namely the neighbourhood system of every point has a countable base, we can state the following theorem, whose proof can be found in \cite[Proposition 8.1]{DM93}.
\begin{theorem}\label{firstaxiomgammaconvergence}
	Let $X$ be a topological space which satisfies the first axiom of countability. We have that $F_\eps\stackrel{\Gamma}{\to}F$ in $X$, as $\eps\to0$, if and only if
	\begin{enumerate}
		\item[$(i)$] For every $x\in X$ and for every sequence $(x_\eps)_\eps\subset X$ converging to $x\in X$, we have
		\begin{align*}
			F(x)\leq\liminf_{\eps\to0}F_\eps(x_\eps);
		\end{align*}
		\item[$(ii)$] For every $x\in X$ there is a sequence $(x_\eps)_\eps\subset X$ converging to $x\in X$ such that
		\begin{align*}
			\limsup_{\eps\to0}F_\eps(x_\eps)\leq F(x).
		\end{align*}
	\end{enumerate}
\end{theorem}
In the present paper, we are working in the framework of Banach spaces endowed with the weak or the $\ast$-weak topology. Since the weak topology is not a metrisable, we need a result that ensures that such a topology is locally metrisable. To this end, we use the following proposition whose proof is contained in \cite[Proposition 2.6]{Del}. 
\begin{proposition} \label{equicoercivenessgamma}
	There exists a distance $d$ on the space of positive measures with the following property. Let $(\mu_\eps)_\eps$ be a sequence of positive measures on $\R^n$ such that $\sup_{\eps>0}\mu_\eps(\R^n)<\infty$. Then, $\mu_\eps\stackrel{\ast}{\to}\mu$ for some positive measure $\mu$ on $\R^n$ if and only if
	\begin{align*}
		\lim_{\eps\to0} d(\mu_\eps,\mu)=0.
	\end{align*} 
\end{proposition}
\begin{definition}
	We say that a sequence of functionals $(F_\eps)_\eps$ is equi-coercive if there exists
	a lower semi-continuous coercive functional $\rho: X\to\R$ such that $F_\eps \geq \rho$ for all $\eps>0$.
\end{definition}
We now state a theorem that ensure the convergence of minimising sequences.
\begin{theorem}\label{convergenzaminimi}
	Let $(F_\eps)_\eps$ be a sequence of is equi-coercive  functions. Then, if $F_\eps\stackrel{\Gamma}{\to}F$, we have that
	\begin{align*}
		\lim_{\eps\to0} \inf_X F_\eps=\min_X F.
	\end{align*}
	Moreover, if $(x_\eps)_\eps\subset X$ is minimising sequence for $F$, then every of its cluster point is a minimum for $F$.
\end{theorem}

\subsection{Convex sub-additive envelope} The following functions are essential in the formulation of the relaxed functional.
\begin{definition}\label{sub}
	A function $\psi:[0,+\infty)\to\R{}$ is said to be \emph{sub-additive} if
	\begin{align*}
		\psi(s+t)\leq\psi(s)+\psi(t),
	\end{align*}
	for any $s,t\geq0$.
\end{definition}
\begin{definition}
	Let $\psi:\R\to\R$. We define its \emph{convex envelope} $\psi^{\text{cvx}}:\R\to\R$ as
	\[
	\psi^{\text{cvx}}(x)\coloneqq\sup\{\rho(x) \, |\, \rho\ \text{is convex and}\ \rho\leq\psi\},
	\]
	for all $x\in\R$.
\end{definition}
\begin{definition}\label{psitilde}
	Let $\psi:[0,+\infty)\to\R{}$. 
	The \emph{convex sub-additive envelope} of $\psi$ is the function
	$\widetilde{\psi}:[0,+\infty)\to\R{}$ defined as
	\begin{align*}
		\widetilde{\psi}(s)\coloneqq\sup\{f(s)\, |\, f:[0,+\infty) \to\R{}\ \emph{is convex, sub-additive and}\ f\leq\psi\},
	\end{align*} 
	for all $s\in[0,+\infty)$.
\end{definition}
\begin{definition}\label{psic}
	Let $\psi:[0,+\infty)\to\R{}$.
	We define the function $\psi^c:[0,+\infty)\to\R{}$ as
	\[
	\psi^c(s)\coloneqq\min\{\widetilde{\psi}(r)+\widetilde{\psi}(t) \,|\, s=r+t\},
	\]
	for all $s\in[0,+\infty)$.
\end{definition}
The main properties of $\widetilde{\psi}$ and $\psi^c$ have been widely explored in \cite{CarCri}, \cite{CarCriDie} and \cite{CriFis}. We just recall some results on the convex sub-additive envelope. The first one is a combination of \cite[Lemma A.11]{CarCriDie} and \cite[Lemma 2.2]{CarCri}.

\begin{lemma}\label{lem:char_psi_s0}
	Let $\psi:[0,+\infty)\to(0,+\infty)$. Then
	\[
	\widetilde{\psi}=\widetilde{\psi^\text{cvx}}.
	\]
	Namely, in order to compute the convex sub-additive envelope of $\psi$, we can assume, without loss of generality, that $\psi$ is convex.\\
	Moreover, assume $\psi$ to be convex. Then, there exists $s_0\in(0,+\infty]$ such that
	\begin{align*}
		\widetilde{\psi}(s)=\begin{cases}
			\psi(s)\quad & s\leq s_0, \\[5pt]
			\theta s \quad& s> s_0,
		\end{cases}
	\end{align*}
	for some $\theta>0$.
\end{lemma}

\begin{remark}
	Note that, if $\psi$ is differentiable at $s_0$, then $\theta=\psi'(s_0)$. In particular, if $s_0<+\infty$, it holds that $\widetilde{\psi}$ is linear in $[s_0,+\infty)$.
\end{remark}
\begin{definition}\label{theta}
	Let $\psi:[0,+\infty)\to\R$. We define the \emph{recession coefficients} of $\widetilde{\psi}$ and $\psi^c$ as
	\begin{align*}
		\widetilde{\theta}\coloneqq\lim_{s\to+\infty}\frac{\widetilde{\psi}(s)}{s}\quad\text{and}\quad
		\theta^c\coloneqq\lim_{s\to+\infty}\frac{\psi^c(s)}{s},
	\end{align*}
	respectively, where $\widetilde{\psi}$ is as in Definition \ref{psitilde} and $\psi^c$ as in Definition \ref{psic}.
\end{definition}
The proof of the following lemma can be found in \cite[Lemma 5.2]{CriFis}.
\begin{lemma}\label{lem:thetac_thetawild}
	Let $\psi:[0,+\infty)\to\R$. Then, $\widetilde{\theta} = \theta^c$. Their common value will be denoted as $\theta$.
\end{lemma}

\subsection{Functions of (pointwise) bounded variation}
In order to define the relaxed functional, we need to introduce some notation about bounded variation functions. More details on the topic can be found in \cite{AFP} and \cite{Leoni_book}. We introduce the following limits
\begin{align*}
	h^-(x)=\liminf_{y\to x} h(y),\quad\quad\quad h^+(x)=\limsup_{y\to x} h(y).
\end{align*}
In particular, if $x\in(a,b)$ is a point of continuity for $h$, then $h^-(x)=h^+(x)=h(x)$. 
\begin{definition}\label{cuts}
	Let $h\in \bv(a,b)$. We call
	\[
	\Gamma_h \coloneqq \left\{(x,y)\in\R^2 \, |\, x\in(a,b),\, h(x)\leq y \leq h^+(x)\right\}
	\]
	the \emph{extended graph} of $h$.
	Moreover, we define: \vspace{0.2cm}
	\begin{itemize}
		\item[$(i)$] The \emph{jump part} of $\Gamma_h$ as
		\[
		\Gamma^j_h\coloneqq\{(x,y)\in\R^2 \, |\, x\in(a,b),\, h^-(x)\leq y<h^+(x)\};
		\]
		\item[$(ii)$]  The \emph{cut part} of $\Gamma_h$ as
		\[
		\Gamma^c_h\coloneqq\{(x,y)\in\R^2 \, |\, x\in(a,b),\, h(x)\leq y<h^-(x)\};
		\]
		\item[$(iii)$]  The \emph{regular part} of $\Gamma_h$ as
		\[
		\Gamma^\rr_h\coloneqq \Gamma_h\setminus (\Gamma_h^j\cup\Gamma_h^c).
		\]
	\end{itemize}
	Moreover, we introduce the notation $\widetilde{\Gamma}_h\coloneqq \Gamma^\rr_h\cup\Gamma^j_h$.
\end{definition}

\begin{remark}
	Note that
	\[
	\Gamma_h=\widetilde{\Gamma}_h\cup \Gamma_h^c
	=\Gamma_h^\rr\cup\Gamma_h^j\cup\Gamma_h^c,
	\]
	holds for every $h\in \bv(a,b)$.
	Moreover, when there is no room for confusion, we will drop the suffix $h$ in the notation above. 
\end{remark}
We now recall two definitions that has been largely used in \cite{CriFis}.
\begin{definition}\label{deltacoverdef}
	Let $h\in\bv(a,b)$, and $\delta>0$. We say that a finite family $(R^j)_{j=1}^N$ of open and pairwise disjoint rectangles is \emph{$\delta$-admissible cover} for $\Gamma$, if
	\begin{itemize}
		\item[(i)] The side lengths of each $R^j$'s is less than $\delta$;
		\item [(ii)]It holds
		\[
		\Gamma \subset \bigcup_{i=1}^N \overline{R}^j;
		\]
		\item[(iii)] $\ho(\Gamma\cap\partial R^j)=0$  for all $j=1,\dots, N$.
	\end{itemize}
\end{definition}
It is possible to show that each bounded variation function admits ad $\delta$-admissible cover.
\begin{definition}\label{gridconstant}
	Let $h\in\bv(a,b)$, and $\delta>0$. A function $u\in L^1(\Gamma)$ is called \emph{$\delta$-grid constant} if there exists a $\delta$-admissible cover for $\Gamma$, such that $u_{|R^j}=u^j\in\R$, for every $j=1,\dots, N$.
	Moreover, we say that $u\in L^1(\Gamma)$ is \emph{grid constant} if there exists $\delta>0$ such that it is $\delta$-grid constant.
\end{definition}

\section{Mathematical setting}\label{mathseting}	

This section is devoted to introduce the setting for the functionals which are object of the study.\\
Let $a<b$ and consider a non-negative function $h:(a,b)\to[0,+\infty)$ such that $h\in\bv(a,b)$. We define the sub-graph of $h$ as
\begin{align*}
	\O\coloneqq\{\bx=(x,y)\in\R^2\, |\, y<h(x)\},
\end{align*}
and its graph as
\begin{align*}
	\Gamma\coloneqq\partial\Omega\cap (a,b)\times\R.
\end{align*}
Let the displacement be a function $v\in H^1(\O;\R^2)$. The starting configuration of the crystalline substrate and that of the deposited film are represented by a function $E_0:\R{} \to \R^{2\times 2}$, defined as
\begin{align*}
	E_0(y)\coloneqq\begin{cases}
		t e_1 \otimes e_1\quad&\text{if}\ y\geq0 ,\\[5pt]
		0\quad&\text{if}\ y<0.
	\end{cases}
\end{align*}
Here, $t>0$ is a constant depending on the lattice of the substrate, and $\{e_1,e_2\}$ is the canonical basis of $\R^2$.
We assume that the crystalline structure of the film and the substrate have a small mismatch, which mathematically can be rephrased as $|t|\ll1$.
Such an assumption allows us to work in the framework of linearised elasticity. 
In particular, the relevant object needed to compute the elastic energy is the symmetric gradient of the displacement
\begin{align*}
	E(v)\coloneqq\frac{1}{2}(\nabla v + \nabla^\top v),
\end{align*}
where $\nabla^\top v$ is the transposed of the matrix $\nabla v$. Finally, we assume that the substrate and the film share similar elastic properties, so they are described by the same positive definite elasticity tensor $\mathbb{C}=(c_{ijnm})_{i,j,n,m=1}^2$. The elastic energy density is given by  $W:\R^{2\times 2}\to \R$, defined as
\begin{align*}
	W(A)\coloneqq\frac{1}{2} A\cdot \mathbb{C}  [A]
	= \frac{1}{2}\sum_{i,j,m,n=1}^2 c_{ijnm} a_{ij}a_{nm},
\end{align*}
for a $2\times 2$ matrix $A=(a_{ij})_{i,j=1}^2$. In addition we require $W$ to be a positive quadratic form, namely
\begin{align}\label{quadraticform}
	W(A)>0,
\end{align}
for every $A\neq0$.\\

 Let the adatom density be a Radon measure $\mu$, that we decompose as
\begin{align*}
	\mu=u\cH^1\llcorner\Gamma+\mu^s,
\end{align*} 
where $\mu^s$ is the singular part of $\mu$ with respect to the measure $\mathcal{H}^1$. 
\subsection{Sharp model.}\label{sharpmodel} Here, we introduce the definitions of the configurations taken into account in our model.
\begin{definition}
	We say that the triplet $(\O,v,\mu)$ is an \emph{admissible sharp configuration} for $\cG$ if $\O$ is the sub-graph of $h\in\bv(a,b)$, $v\in H^1(\O;\R^2)$ with $v=0$ on $Q\setminus \O$, and $\mu=u\cH^1\llcorner\Gamma+\mu^s$, with $u\in L^1(\Gamma)$. We denote the set of admissible sharp configurations by $\cA$. 
\end{definition}
\begin{definition}
	An admissible sharp configuration $(\O,v,\mu)$ is called \emph{regular} if $\O$ is the sub-graph of a Lipschitz function and $\mu=u\cH^1\llcorner\Gamma$ with $u\in L^1(\Gamma)$. The set of such configurations is denoted by $\cA_r$.
\end{definition}
The next definition introduces configurations which satisfies the mass constraints.
\begin{definition}
	Given $M,m>0$, we say that $(\O,v,\mu)\in \cA(m,M)$ if $(\O,v,\mu)\in\cA$ and
	\begin{align}\label{constraints}
		\int_a^b h(x)\ \de x=M\quad\text{and}\quad \mu(\R^2)=m.
	\end{align}
	In a similar way we say that $(\O,v,\mu)\in\cA_r(m,M)$ if $(\O,v,\mu)\in\cA_r$ and \eqref{constraints} holds.
\end{definition}
Now, we are in position to give the definitions of the functional treated in \cite{CriFis}. Let  $\psi:[0,\infty)\to(0,\infty)$ be a Borel function with $\inf_{s\geq0}\psi(s)>0$.
\begin{definition}\label{hfunctional}
We define the functional $\cH:\cA\to[0,+\infty]$ as
	\begin{align*}
		\cH(\O,v,\mu)\coloneqq\begin{cases}
			\displaystyle\int_{\O}W\big(E(v)-E_0(y)\big)\ \de \textbf{x}+\int_{\Gamma}\psi(u)\dho\quad&(\O,v,\mu)\in\cA_r,\\[10pt]
			+\infty\quad&\text{else}.
		\end{cases}
	\end{align*}
and $\cH^{m,M}:\cA\to[0,+\infty]$ as
	\begin{align*}
	\cH^{m,M}(\O,v,\mu)\coloneqq\begin{cases}
		\displaystyle\int_{\O}W\big(E(v)-E_0(y)\big)\ \de \textbf{x}+\int_{\Gamma}\psi(u)\dho\quad&(\O,v,\mu)\in\cA_r(m,M),\\[10pt]
		+\infty\quad&\text{else}.
	\end{cases}
		\end{align*}
\end{definition}
\begin{definition}\label{Ffunctional}
We define the functional $\cF:\cA\to[0,+\infty]$ as
\begin{align*}
	\cF(\O,v,\mu)\coloneqq\begin{cases}
		\cG(\O,v,\mu)\quad&\text{if } (\O,v,\mu)\in\cA,\\[5pt]
		+\infty\quad&\text{else},
	\end{cases}
\end{align*}
where
\begin{align*}
	\cG(\O,v,\mu)\coloneqq\int_{\O}W\big(E(v)-E_0(y)\big)\ \de \textbf{x}+\int_{\widetilde{\Gamma}}\widetilde{\psi}(u)\dho+\int_{\Gamma^c}\psi^c(u)\dho+\theta\mu^s(\R^2).
\end{align*}
Moreover, define the functional $\cF^{m,M}:\cA\to[0,+\infty]$ as
\begin{align*}
	\cF^{m,M}(\O,v,\mu)\coloneqq\begin{cases}
		\cG(\O,v,\mu)\quad&\text{if } (\O,v,\mu)\in\cA(m,M),\\[5pt]
		+\infty\quad&\text{else},
	\end{cases}
\end{align*}
\end{definition} 
Throughout the entire paper we will use the convergence sequences of the form $(\eps_n)_{n\in\N}$. Every time we would write $\eps_n\to0$ as $n\to\infty$, in order to enlighten the notation, we write $\eps\to0$ instead. However, since the main results of this paper do not depend on a subsequence, we can replace the discrete parameter with a continuous one and get the same results. \\

We recall the definition of Hausdorff distance.
\begin{definition}
	Let $E,F\subset\R^N$. We define
	\[
	\mathrm{d}_H(E,F)\coloneqq \inf\{ r>0 \,|\, E\subset F_r,\, F\subset E_r \},
	\]
	where, for $A\subset\R^N$ and $r>0$, we set $A_r\coloneqq \{ x+y \,|\, x\in A,\  y\in B_r(0) \}$.
	Moreover, we say that a sequence of compact sets $(E_\eps)_\eps\subset\R^N$ \emph{Hausdorff converges} to a set $E\subset\R^N$, and we write $E_\eps\stackrel{H}{\rightarrow} E$, if $\mathrm{d}_H(E_\eps,E)\to0$ as $\eps\to0$.
\end{definition}
The relaxation result in \cite{CriFis} is made with respect the topology induced by the following definition of convergence. 
\begin{definition}\label{topologyoldpaper}
	 Let $(\O,v,\mu)\in \cA$ and consider a sequence $(\O_\eps,v_\eps,\mu_\eps)_\eps\subset\cA$. We say that $(\O_\eps,v_\eps,\mu_\eps)\to(\O,v,\mu)$, as $\eps\to0$, if
\begin{enumerate}
\item[$(i)$]  $\R^2\setminus \O_\eps\stackrel{H}{\to} \R^2\setminus \O$ in the Hausdorff distance;
\item[$(ii)$] $v_\eps\wto v$ weakly in $H^1_{\text{loc}}(\O;\R^2)$;
\item[$(iii)$] $\mu_\eps\wsto\mu$ in the $\ast$-weak topology.
\end{enumerate}
\end{definition}
\begin{remark}
The Hausdorff convergence $\R^2\setminus \O_\eps\stackrel{H}{\to} \R^2\setminus \O$ is made in such a way that internal fractures are visible inside the film. If we choose $\O_\eps\stackrel{H}{\to}\O$,  we would obtain vertical cuts outside the crystal, which has no physical meaning.
\end{remark}
The main result achieved in \cite{CriFis} is the following.
\begin{theorem} The following statements hold true.
\begin{enumerate}
\item[$(i)$] The functional $\cF$ is the relaxation of $\cH$ with respect to the topology in Definition \ref{topologyoldpaper};
\item[$(ii)$] The functional $\cF^{m,M}$ is the relaxation of $\cH^{m,M}$ with respect to the topology in Definition \ref{topologyoldpaper}.
\end{enumerate}
\end{theorem}
\subsection{Phase-field model.}\label{phasefieldsubsection} We now introduce the phase-field functional that will approximate $\cG$ in terms of the $\Gamma-$convergence. \\
First, we choose a double-well potential $P$, which detects whether we are o not on a cut of $h$.
\begin{definition}\label{potential}
	Let $P:\R{}\to\R^+$ be a continuous function such that
\begin{enumerate}
	\item[$(i)$]  $P^{-1}(0)=\{0,1\}$;
	\item[$(ii)$] There exist $r,C>0$ such that
	\begin{align*}
	P(t)\geq C\abs{t},
	\end{align*}
	for every $\abs{t}>r$.
\end{enumerate} 
\end{definition}
Consider
\begin{align*}
	Q\coloneqq(a,b)\times\R{}\quad\text{and}\quad
	Q^+\coloneqq Q\cap \{y>0\}.
\end{align*}
In the following definition, we define the phase field variable $w\in H^1(Q^+;[0,1])$ and the admissible configurations that will be used in the phase-field formulation.
\begin{definition}
	We say that the triplet $(w,v,u)$ is an \emph{admissible phase-field configuration} if $w\in H^1(Q^+)$ is such that $0\leq w\leq 1$, $v\in H^1(Q;\R^2)$ and $u\in L^1(\R^2)$. We denote the set of admissible phase-field configurations by $\cA_p$. \\
	Given $M,m>0$, we say that $(w,v,u)\in \cA_p(m,M)$ if $(w,v,u)\in\cA_p$ and
	\begin{align*}
		\int_{Q^+} w\ \de \bfx=M\quad\text{and}\quad \int_{Q^+}u\ \de\bfx=m.
	\end{align*}
\end{definition}
The phase-field functional is given in the following definition.
\begin{definition}
We define the functional $\cF_\eps:H^1(Q^+)\times H^1(Q;\R^2)\times L^1(\R^2)\to[0,+\infty]$ as
\begin{align*}
	\cF_\eps(w,v,u)\coloneqq\begin{cases}
		\cG_\eps(w,v,u)\quad&\text{if } (w,v,u)\in\cA_p,\\[5pt]
		+\infty\quad&\text{else},
	\end{cases}
\end{align*}
where
	\begin{align*}
		\cG_\eps(w,v,u)\coloneqq\int_{Q^+}(w+\eta_\eps) W\big(E(v)-E_0(y)\big)\ \de\textbf{x}
		+	\frac{1}{\sigma}\int_{Q^+} \Big[\eps|\nabla w|^2+\frac{1}{\eps}P\big(w\big)\Big]\psi(u)\ \de\textbf{x},
\end{align*}
with $\eta_\eps=o(\eps)$, as $\eps\to0$, and
\begin{align*}
	\sigma\coloneqq2\int_{0}^{1}\sqrt{P(t)}\ \de t. 
\end{align*}
Additionally, we define the functional with the mass constraint  $\cF^{m,M}_\eps:H^1(Q^+)\times H^1(Q;\R^2)\times L^1(\R^2)\to[0,+\infty]$, as
\begin{align*}
	\cF^{m,M}_\eps(w,v,u)\coloneqq\begin{cases}
		\cG_\eps(w,v,u)\quad&\text{if } (w,v,u)\in\cA_p(m,M),\\[5pt]
		+\infty\quad&\text{else}.
	\end{cases}
\end{align*}
\end{definition}


\section{Main result}
In this section, we state the main results. First, we need introduce the notion of convergence we will use.
\begin{definition}\label{topology}
	We say that $(w_\eps,v_\eps,u_\eps)_\eps\subset \cA_p$ converges to $(\O,v,\mu)\in \cA$, and we write $(w_\eps,v_\eps,u_\eps)\to(\O,v,\mu)$ as $\eps\to0$, if
	\begin{enumerate}
		\item[$(i)$] $w_\eps\to \chi_\O$ in $L^1_\loc(Q^+)$, where $\chi_\O$ is the indicator function of $\O$;
		\item[$(ii)$] $v_\eps\to v$ in $L^2_\loc(Q)$;
		\item[$(iii)$] $\mu_\eps\wsto \mu$ $\ast$-weak in the sense of measures,
		where
		\begin{align*}
			\mu_\eps\coloneqq \frac{u_\eps}{\sigma}\Big( \varepsilon |\nabla w_\eps|^2+\frac{1}{\varepsilon} P(w_\eps)  \Big)\mathcal{L}^2\llcorner Q^+.
		\end{align*}
	\end{enumerate}
\end{definition}
	\begin{remark}
	Thanks to Proposition \ref{equicoercivenessgamma}, the configuration space $\cA_p$ is a metric space. 
\end{remark}
The main results of this paper are the following.
\begin{theorem}\label{mainresult1}
 $\cF_\eps\stackrel{\Gamma}{\longrightarrow}\cF$, as $\eps\to0$, with respect to the topology in Definition \ref{topology};
\end{theorem}
The class of configurations for which we have compactness is explained in the following theorem.
\begin{theorem}\label{compactness}
	Let $(w_\eps,v_\eps,u_\eps)_\eps\subset \cA_p(m,M)$ be a sequence such that
	\begin{align*}
		&\sup_{\eps>0}\ 	\cF_\eps(w_\eps,v_\eps,u_\eps)<\infty,\\[5pt]	
		&\sup_{\eps>0}\int_{Q^+} \abs{E(v_\eps)}^2\ \de\bfx<\infty,\\[5pt]
		&\nonumber\mu_\eps\coloneqq u_\eps\cL^2\llcorner Q^+\stackrel{\ast}{\wto} \mu,
	\end{align*}
	as $\eps\to0$, with
	\begin{align*}
		\supp \mu\subset \Gamma.
	\end{align*} 
	Then, there exists $( \O,v,\mu)\in\cA(m,M)$ such that 
	\begin{align*}
		(w_\eps,v_\eps,u_\eps)\to(\O,v,\mu),
	\end{align*}
	as $\eps\to0$.
\end{theorem}
\begin{remark}
	We notice that the function $\psi$ does not need any further hypotheses for the coerciveness, beside being a strictly positive Borel function.
\end{remark}
The following theorem is similar to Theorem \ref{mainresult1}, but when the constraints \eqref{constraints} on the configurations are in force.
\begin{theorem}\label{mainresult2}
$\cF^{m,M}_\eps\stackrel{\Gamma}{\longrightarrow}\cF^{m,M}$, as $\eps\to0$, with respect to the topology in Definition \ref{topology}.
\end{theorem}
The proofs of Theorem \ref{mainresult1} and \ref{mainresult2} are a consequence of the following two steps.
\begin{enumerate}
	\item[Step 1] \textit{Liminf inequality.} In Theorem \ref{liminftheorem} we prove that for every $(\O,v,\mu)\in \cA$ and for every $(w_\eps,v_\eps,u_\eps)_\eps\subset \cA_p$, such that $(w_\eps,v_\eps,u_\eps)\to(\O,v,\mu)$ as $\eps\to0$, we have
	\begin{align*}
		\cF(\O,v,\mu)\leq\liminf_{\eps\to0} \cF_\eps(w_\eps,v_\eps,u_\eps).
	\end{align*}
	\item[Step 2] \textit{Limsup inequality for the constrained problem.} In Theorem \ref{limsup} we prove that for every $(\O,v,\mu)\in\cA(m,M)$, there is a sequence $(w_\eps,v_\eps,u_\eps)_\eps\subset \cA_p(m,M)$ such that $(w_\eps,v_\eps,u_\eps)\to (\O,v,\mu)$ as $\eps\to0$ and
	\begin{align*}
		\limsup_{\eps\to0} \cF(w_\eps,v_\eps,u_\eps)\leq \cF (\O,v,\mu).
		\end{align*}
\end{enumerate}
The following theorem is a consequence of Theorems \ref{convergenzaminimi},  \ref{compactness} and \ref{mainresult2} and states the convergence of the minimisation sequences for $\cF_\eps$. We remark that the family $(\cF_\eps)_\eps$ is equi-coercive and lower semi-continuous. 
\begin{theorem}
	We have that
	\begin{align*}
	\lim_{\eps\to0}\min_{\cA_p(m,M)}\cF_\eps=\min_{\cA(m,M)} \cF.
	\end{align*}
\end{theorem}
We would like to briefly give an idea of the proof of Theorems \ref{liminftheorem} and \ref{limsup}. \\
The main difficulty of the liminf inequality comes from the surface term. Indeed, for a given configuration $(\O, v,\mu)\in\cA$, the set $\Gamma$ might present a dense cut set. To get around it, we define 
\begin{align*}
	C_\xi\coloneqq \{\textbf{x}\in\Gamma^c:h^-(x)-y<\xi\}.
\end{align*} 
It is possible to prove that $\Gamma^c\setminus C_\xi$ is a finite number of vertical segments. Since we can now separate those cuts and include each of them in a suitable rectangle $R$, we suppose that we only have one cut point $x^c$ and we repeat the argument for each other one afterward.  The energy around $x^c$ is given by
\begin{align*}
	\liminf_{\eps\to0}&\frac{1}{\sigma}\int_{Q^+\cap R} \Big[\eps|\nabla w_\eps|^2+\frac{1}{\eps}P(w_\eps)\Big]\psi(u_\eps)\ \de\textbf{x}.
\end{align*}
From here we can separate the energy on the left and on the right of each single cut, namely we can split $R=R^\ell\cup R^\rr$ in such a way that $R^\ell\cap R^\rr=\{x^c\}\times [0,h^-(x^c)]$. Therefore, when we pass to the liminf, we get
\begin{align*}
	\nonumber	\liminf_{\eps\to0}\frac{1}{\sigma}\int_{Q^+\cap R} &\Big[\eps|\nabla w_\eps|^2+\frac{1}{\eps}P(w_\eps)\Big]\psi(u_\eps)\ \de\textbf{x}\\[5pt]
	\nonumber\geq\liminf_{\eps\to0}&\frac{1}{\sigma}\int_{Q^+\cap R^\ell} \Big[\eps|\nabla w_\eps|^2+\frac{1}{\eps}P(w_\eps)\Big]\psi(u_\eps)\ \de\textbf{x}\\[5pt]
	&\hspace{0.5cm}+\liminf_{\eps\to0}\frac{1}{\sigma}\int_{Q^+\cap R^\rr} \Big[\eps|\nabla w_\eps|^2+\frac{1}{\eps}P(w_\eps)\Big]\psi(u_\eps)\ \de\textbf{x}.
\end{align*}
First, we prove that
\begin{align*}
	\frac{u_\eps}{\sigma}\Big( \varepsilon |\nabla w_\eps|^2+\frac{1}{\varepsilon} P(w_\eps)  \Big)\mathcal{L}^2\llcorner R^\ell&\wsto f\ho\llcorner  (\Gamma^c \cap R)+(\mu^s)^\ell,\\[5pt]
	\frac{u_\eps}{\sigma}\Big( \varepsilon |\nabla w_\eps|^2+\frac{1}{\varepsilon} P(w_\eps)  \Big)\mathcal{L}^2\llcorner R^\rr&\wsto g\ho\llcorner (\Gamma^c \cap R)+(\mu^s)^\rr,
\end{align*}
for some $f,g\in L^1(\Gamma^c\setminus C_\xi)$, as $\eps\to0$ such that 
\begin{align}\label{f+g}
	f+g=u_{|\Gamma^c\setminus C_\xi}\quad\text{and}\quad (\mu^s)^\ell+(\mu^s)^\rr=\mu^s.
\end{align}
Therefore, if we take into account the rectangle on the left of the cut, we can prove that
\begin{align*}
	\liminf_{\eps\to0}\frac{1}{\sigma}\int_{Q^+\cap R^\ell} \Big[\eps|\nabla w_\eps|^2&+\frac{1}{\eps}P(w_\eps)\Big]\psi(u_\eps)\ \de\textbf{x}\\[5pt]&\geq \int_{\partial^\ast R^\ell}\widetilde{\psi}(f)\ \de\ho+\theta (\mu^s)^\ell(\R^2).
\end{align*}
We can conclude by summing up the contribution given by the term integrated on $R^r$ and by considering the definition of $\psi^c$, given the fact that
\begin{align*}
	\widetilde{\psi}(f)+\widetilde{\psi}(g)\geq \psi^c(u),
\end{align*}
provided that $f+g=u$.\\

The limsup inequality is based on a diagonalisation argument. Namely, we can approximate a given configuration $(\O,v,\mu)\in\cA(m,M)$, with a sequence $(\O_\eps,v_\eps,u_\eps\ho\llcorner\Gamma_\eps)$ such that $\O_\eps$ is the sub-graph of a $\cC^\infty$ function and $u_\eps\in L^1(\Gamma_\eps)$ is piece-wise constant in a suitable sense (see Definition \ref{gridconstant} and Theorems \ref{limsuptheorem} and \ref{wrigglingthm}).\\
One of the key ingredients of the proof of the Theorem \ref{wrigglingthm} is the so called \textit{wriggling process}, first introduced in \cite{CarCriDie}, largely used in \cite{CarCri} and later refined in \cite{CriFis}. The idea relies on the fact that $\psi$ and $\widetilde{\psi}$ agree on $[0,s_0)$, for some $s_0\in(0,+\infty]$. In particular, if $s_0<+\infty$ the function $\widetilde{\psi}$ is linear on $(s_0,+\infty)$ (see Lemma \ref{lem:char_psi_s0}). Thus, if $u\in L^1(\Gamma)$ is the density we would like to approximate, in case $u\leq s_0$, we define the approximant profile $h_\eps$ as $h$ and the approximant density $u_\eps\in L^1(\Gamma_\eps)$ as $u$. In case $u>s_0$, the surface energy in the recovery sequence can be considered constant on suitable rectangles (see Definition \ref{gridconstant} and \cite[Proposition 7.6]{CriFis}). If $R$ is one of those rectangles, we have that the surface energy is given by $\widetilde{\psi}(u)\ho(\Gamma\cap R)$. The idea is to write
\[
\widetilde{\psi}(u)\ho(\Gamma\cap R)
= \widetilde{\psi}(r s_0)\ho(\Gamma\cap R)
= r \widetilde{\psi}(s_0)\ho(\Gamma\cap R)
= \psi(s_0) \left[ r\ho(\Gamma\cap R)\right],
\]
for some $r>1$, where in the last step we used the fact that $\psi(s_0)=\widetilde{\psi}(s_0)$.\\
In Theorem \ref{realtheoremlimsup}, we use a similar strategy to the one used in \cite[Theorem 3.1]{BonCha02}, in which the phase-field approximating sequence $(w_\eps)_\eps$ is built by making use of the \textit{almost} optimal profile problem as explained in \cite[Proposition 2]{Mod87}, which is the solution of the following differential equation,
\begin{equation}
	\begin{cases}
		\eps^2	\abs{\gamma'_\eps(t)}^2=P(\gamma_\eps(t))+\sqrt{\eps}\quad t\in\R,\\[5pt]
		\gamma_\eps(0)=0,\\[5pt]
		\gamma_\eps(1)=1.
	\end{cases}
	\label{optimalprofile1}
\end{equation}
We can define the phase-field varaible as
\begin{align*}
	w_\eps(\bfx)\coloneqq \gamma_\eps\Big(\frac{d_\O(\bfx)}{\eps}\Big),
\end{align*}
where $\gamma_\eps$ is the solution of \eqref{optimalprofile1} and 
\begin{align*}
	d_\O(\bfx)\coloneqq \dist{\bfx}{\O}-\dist{\bfx}{\R^2\setminus \O},
\end{align*}
is the signed distance from $\O$. Now, once we make sure that all the constraints are satisfied, we follow a path inspired to the one contained in \cite{BonCha02}. 
\section{Compactness}
In this section we give the proof of Theorem \ref{compactness}. For the reader's convenience, we report the statement.
\begin{theorem}
	Let $(w_\eps,v_\eps,u_\eps)_\eps\subset \cA_p(m,M)$ be a sequence such that
	\begin{align}
		&\label{compactnessbounds1}\sup_{\eps>0}\ 	\cF_\eps(w_\eps,v_\eps,u_\eps)<\infty,\\[5pt]	
		&\label{compactnessbounds2}\sup_{\eps>0}\int_{Q^+} \abs{E(v_\eps)}^2\ \de\bfx<\infty,\\[5pt]
		&\nonumber\mu_\eps\coloneqq u_\eps\cL^2\llcorner Q^+\stackrel{\ast}{\wto} \mu,
	\end{align}
as $\eps\to0$, with
	\begin{align}\label{compactnessbounds4}
		\supp \mu\subset \Gamma.
	\end{align} 
	Then, there exists $( \O,v,\mu)\in\cA(m,M)$ such that 
	\begin{align*}
		(w_\eps,v_\eps,u_\eps)\to(\O,v,\mu),
	\end{align*}
	as $\eps\to0$.
\end{theorem}
\begin{proof}
	\emph{Step 1.} The compactness of the phase-field variable $w_\eps$ follows a standard argument, which makes use of the first bound in \eqref{compactnessbounds1}, and can be found in \cite{Mod87}. Therefore, we have the existence of $w\in \bv(Q^+;\{0,1\})$ such that, up to a subsequence $w_\eps\to \chi_\O$ in $L^1(Q^+)$. Moreover
	\begin{align*}
		\int_{Q^+} w_\eps\ \de\bfx\to \int_{Q^+}\chi_\O(\bfx)\ \de\bfx=\abs{\O}=M,
	\end{align*} 
	thus the mass constraint is preserved.\\
	\emph{Step 2.} From \eqref{compactnessbounds4} we have that $\supp\mu\subset\Gamma$. From the fact that $\partial\Gamma \cap (a,b)=\emptyset$, we have, by standard properties of the $\ast$-weak convergence that
	\begin{align}\label{massconstrmu}
		\mu(Q^+)=\lim_{\eps\to0} \mu_\eps(Q^+)=m.
	\end{align}
	\emph{Step 3.} By \eqref{compactnessbounds2}, there exists a function $\overline{E}\in L^2(Q;\R^2)$ such that
	\begin{align*}
		E(v_\eps)\wto \overline{E}
	\end{align*}
	in $L^2(Q;\R^2)$ as $\eps\to0$. By Korn's inequality, there exists $C>0$ such that
	\begin{align*}
		\int_{Q^+}\abs{\nabla v_\eps(\bfx)}^2\ \de\bfx \leq C \int_{Q^+} \abs{E(v_\eps)}^2\ \de\bfx <\infty.
	\end{align*}
	From that, again by compactness, we have the existence of a function $v\in H^1(Q;\R^2)$ such that $v_\eps\wto v$ as $\eps\to0$. By uniqueness of the weak limit, $E(v_\eps) \wto \overline{E}=E(v)$.\\
	By putting together the three above steps, we conclude proof of the Theorem.
\end{proof}
\begin{remark}
	In Theorem \ref{compactness}, hypothesis \eqref{compactnessbounds4} is essential. Indeed, if we drop it we cannot ensure the validity of \eqref{massconstrmu} and neither the fact that $(\mu_\eps)_\eps$ is converging to a Radon measure that is supported on $\Gamma$. This behaviour is due to the fact that the phase-field variable and the density one are completely independent (see Figure \ref{figcompactness}).
\end{remark}
\begin{figure}
	\begin{center}
		\includegraphics[scale=0.35]{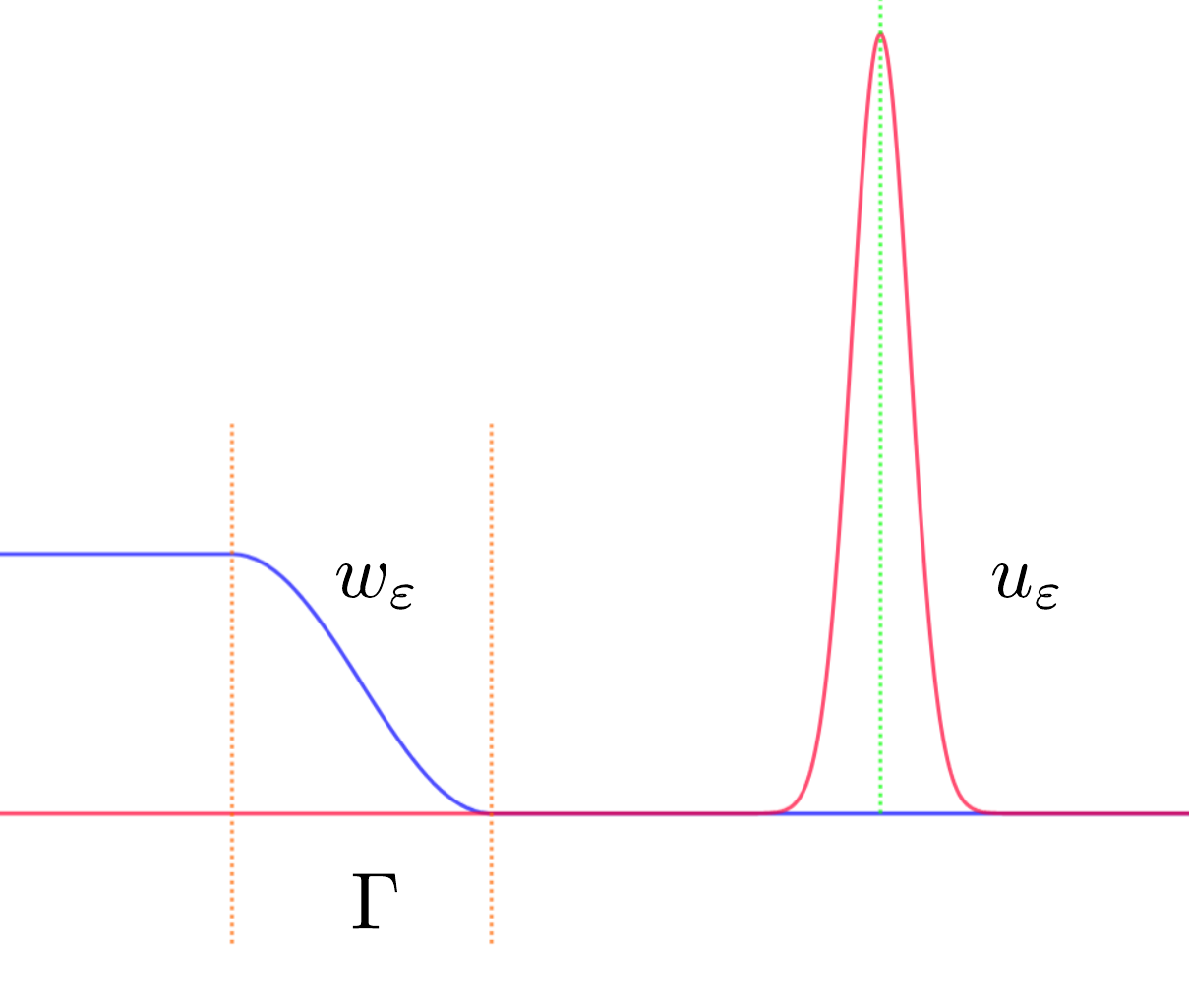}
		\caption{The phase-field variable converges to $\chi_\O$, whereas the $u_\eps$, being independent from the other variables, can concentrate in an area away from $\Gamma$.}
		\label{figcompactness}
	\end{center}
\end{figure}
\section{Liminf inequality}
The liminf inequality relies on the strategy used in \cite[Theorem 6.1]{CriFis} and the following theorem, which is a corollary of the main result of \cite{CarCri}. In this paper, the authors  work with a surface energy in a bounded domain in $\R^n$, with an adatom density at the interface. We report such a result, with our notation. The proof is presented in \cite[Theorem 3.10]{CarCri}.
\begin{theorem}\label{cctheorem}
	Let $\Sigma\subset\R^2$ be an open set and consider the functional
	\begin{align*}
		\cE_\eps(w,u)\coloneqq \frac{1}{\sigma}\int_\Sigma \Big[\eps\abs{\nabla w}^2+\frac{1}{\eps}P(w)\Big]\psi(u)\ \de\bfx,
	\end{align*}
	where $P$ as in Definition \ref{potential} and, with an abuse of notation, $(w,u)\in \cA_p$. Then, $\cE_\eps\stackrel{\Gamma}{\to}\cE$, as $\eps\to0$, where
	\begin{align}\label{GammalimitCC}
		\cE(A,\mu)\coloneqq \int_{\partial^\ast A\cap \Sigma } \widetilde{\psi} (u)\dho+\theta\mu^s(\R^2),
	\end{align}
for $(A,\mu)\in\cA$, with respect the topology in Definition \ref{topology}. Here $\theta$ is as in Definition \ref{theta} and $\mu^s$ is the singular part of $\mu$ with respect to $\ho$.
\end{theorem}
\begin{remark}\label{liminfsolution}
	We notice that if we consider an open subset $B$ compactly contained in $\Sigma$ and denote by $\cC_b(B)$ the set of continuous and bounded functions on $B$, then, if we consider the convergence 
	\begin{align*}
		\int_B \varphi\ \de\mu_\eps\to\int_B \varphi\ \de\mu,
	\end{align*} 
	for every $\varphi\in\cC_b(B)$, we have that the set of integration in the $\Gamma$-limit in \eqref{GammalimitCC} is $\partial^\ast A\cap\overline{B}$.
\end{remark}
The main theorem of this section, namely the liminf inequality, is the following.
\begin{theorem}\label{liminftheorem}
	Let $(\O,v,\mu)\in \cA$. Then, for every $(w_\eps,v_\eps,u_\eps)_\eps\subset \cA_p$, such that $(w_\eps,v_\eps,u_\eps)\to(\O,v,\mu)$ as $\eps\to0$, we have
	\begin{align*}
		\cF(\O,v,\mu)\leq\liminf_{\eps\to0} \cF_\eps(w_\eps,v_\eps,u_\eps).
	\end{align*}
\end{theorem}
\begin{proof}\emph{Step 1. Bulk term.} Take any compact set $K\subset\subset \Omega$. We have
	\begin{align}\label{bulktermiliminf}
		\nonumber	\liminf_{\eps\to0}\int_{Q^+} (w_\eps+\eta_\eps)W\big(E(v_\eps)-E_0(y)\big)\ \de \textbf{x}&\geq \liminf_{\eps\to0}\int_{K} (w_\eps+\eta_\eps)W\big(E(v_\eps)-E_0(y)\big)\ \de \textbf{x}\\[5pt]
		&\geq \int_KW\big(E(v)-E_0(y)\big)\ \de \textbf{x},
	\end{align}
	where in the last step we used the fact that $W$ is convex and the fact that $\liminf_{\eps\to0}w_\eps=1$ on every $K\subset\subset \Omega$. Now we can consider an increasing sequence of compact sets $K_j\subset\subset\O$ and conclude with the use of the Dominated Convergence Theorem.  Here we remark that, differently from \cite[Theorem 6.1]{CriFis}, the $v_\eps$'s are defined on the entire $Q$ and there is no need of the additional technicalities as in the mentioned Theorem, in which the displacement sequence was defined only in $H^1_\loc(\O_\eps;\R^2)$.\\
	
	\emph{Step 2. Surface term}. We would like to separate $\widetilde{\Gamma}$ and $\Gamma^c$. On the regular part we apply the result contained in \cite{CarCri} together with an error that  corresponds when the cut part meet the regular part. On the cut part, we isolate each vertical cut, far enough from the regular part, and we look at the contribution given by the left and right contribution given by $u_\eps$. \\
	Fix $\xi>0$ and consider the set
	\begin{align*}
		C_\xi\coloneqq \{\textbf{x}=(x,y)\in\Gamma^c:h^-(x)-y<\xi\}.
	\end{align*}
	By a standard measure theory argument, there is a sequence $(\xi_\gamma)_\gamma\subset\R$ such that $\xi_\gamma\to0$, as $\gamma\to0$ and
	\begin{align}\label{cxi}
		\mu(\Gamma\cap\partial C_{\xi_\gamma})=0,
	\end{align} 
	for every $\gamma>0$.
	As a consequence, from \cite[Lemma 5.3]{CriFis}  and for every $\gamma>0$, we have that $\Gamma^c\setminus C_{\xi_\gamma}$ consists of a finite number of vertical segments, whose projections  on the $x$-axes corresponds to the set $(x^i)_{i=1}^N$. Recalling the definition of $\Gamma^c$ (see Definition \ref{cuts}), it holds that $C_{\xi_\gamma}$ is monotonically converging to the empty set, as $\eps\to0$. Therefore, we get that
	\begin{equation}
		\mu(C_{\xi_\gamma})\to0,\quad\quad\quad
		\mu(\Gamma^c\setminus C_{\xi_\gamma})\to \mu(\Gamma^c),
	\end{equation}
	as $\gamma\to0$.
	Let $\delta=\delta(\eps)>0$ such that $\delta\to 0$ as $\eps\to0$ and we have $\delta<|x^i-x^j|$, for every $i,j=1,\dots,N$.
	As we have a finite number of cuts, in order to simplify the notation, we do the following construction as we had only one cut point, and then we repeat it for each other one. Therefore, let $(x^c,h(x^c))$ be the only cut point of $\Gamma$. 
	Consider the rectangle
	\begin{align*}
		R=R(\delta,\xi_\gamma)\coloneqq(x^c-\delta,x^c+\delta)\times (-\delta,h^-(x^c)-\xi_\gamma).
	\end{align*}
	Up to further reducing $\delta$, we can assume that $\widetilde{\Gamma}\cap R=\emptyset$. We can split $R$ as
	\begin{align*}
		R^\ell&=R^\ell(\delta,\xi_\gamma)\coloneqq(x^c-\delta,x^c)\times (-\delta,h^-(x^c)-\xi_\gamma),\\[5pt]
		R^\rr&=R^\rr(\delta,\xi_\gamma)\coloneqq(x^c-\delta,x^c)\times (-\delta,h^-(x^c)-\xi_\gamma).
	\end{align*}
	We remark that we need to consider rectangles that go below $\{y=0\}$, as a cut might touch the $y$-axes and a singular measure (e.g. a Dirac delta) might be present at the endpoint of such a cut. \\
	Therefore, we get
	\begin{align}\label{surfaceenergyR}
		\nonumber\liminf_{\eps\to0} \frac{1}{\sigma}\int_{Q^+} \Big[\eps|\nabla w_\eps|^2+\frac{1}{\eps}P(w_\eps)\Big]&\psi(u_\eps)\ \de\textbf{x}\geq \liminf_{\eps\to0}\frac{1}{\sigma}\int_{Q^+\cap R} \Big[\eps|\nabla w_\eps|^2+\frac{1}{\eps}P(w_\eps)\Big]\psi(u_\eps)\ \de\textbf{x}\\[5pt]
		&+\liminf_{\eps\to0}\frac{1}{\sigma}\int_{Q^+\setminus R} \Big[\eps|\nabla w_\eps|^2+\frac{1}{\eps}P(w_\eps)\Big]\psi(u_\eps)\ \de\textbf{x}.
	\end{align}
	\textit{Step 2.1 Cut part $\Gamma^c$.}  We deal now with the first term on the right-hand side of \eqref{surfaceenergyR}. We have
	\begin{align}\label{CCstep}
		\nonumber	\liminf_{\eps\to0}\frac{1}{\sigma}\int_{Q^+\cap R} \Big[\eps|\nabla w_\eps|^2+\frac{1}{\eps}P(w_\eps)\Big]&\psi(u_\eps)\ \de\textbf{x}
		\geq\liminf_{\eps\to0}\frac{1}{\sigma}\int_{Q^+\cap R^\ell} \Big[\eps|\nabla w_\eps|^2+\frac{1}{\eps}P(w_\eps)\Big]\psi(u_\eps)\ \de\textbf{x}\\[5pt]
		&+\liminf_{\eps\to0}\frac{1}{\sigma}\int_{Q^+\cap R^\rr} \Big[\eps|\nabla w_\eps|^2+\frac{1}{\eps}P(w_\eps)\Big]\psi(u_\eps)\ \de\textbf{x}.
	\end{align}
	Now, we localise the $\ast$-weak convergence of $u_\eps$ to $\mu$. Namely, we prove
	\begin{align*}
		\frac{u_\eps}{\sigma}\Big( \varepsilon |\nabla w_\varepsilon|^2+\frac{1}{\varepsilon} P(w_\varepsilon)  \Big)\mathcal{L}^2 \llcorner R^\ell&\wsto   f\mathcal{H}^1\llcorner (\Gamma^c\cap R) + (\mu^s)^\ell,\\[5pt]
		\frac{u_\eps}{\sigma}\Big( \varepsilon |\nabla w_\varepsilon|^2+\frac{1}{\varepsilon} P(w_\varepsilon)  \Big)\mathcal{L}^2 \llcorner R^\rr&\wsto g\mathcal{H}^1\llcorner (\Gamma^c\cap R) + (\mu^s)^\rr,
	\end{align*}
	for some $f,g\in L^1(\Gamma^c\setminus C_{\xi_\gamma})$, as $\eps\to0$ such that 
	\begin{align*}
		f+g=u_{|\Gamma^c\setminus C_{\xi_\gamma}}\quad\text{and}\quad (\mu^s)^\ell+(\mu^s)^\rr=\mu^s,
	\end{align*}
	where $(\mu^\ell)^s$ and $(\mu^\rr)^s$ are supported in $\Gamma^c\setminus C_{\xi_\gamma}$. Since the strategy is similar to the one proposed in \cite[Theorem 6.1]{CriFis}, we only report the main ideas. By the $\ast$-weak covergence of $u_\eps$ to $\mu$ and by compactness we have that up to a subsequence
	\begin{align*}
		\frac{u_\eps}{\sigma}\Big( \varepsilon |\nabla w_\varepsilon|^2+\frac{1}{\varepsilon} P(w_\varepsilon) \Big)\mathcal{L}^2 \llcorner R^\ell\stackrel{\ast}{\wto}\mu^\ell
	\end{align*}
	and
	\begin{align*}
		\frac{u_\eps}{\sigma}\Big( \varepsilon |\nabla w_\varepsilon|^2+\frac{1}{\varepsilon} P(w_\varepsilon)  \Big)\mathcal{L}^2 \llcorner R^\rr\stackrel{\ast}{\wto}\mu^\rr,
	\end{align*}
	for some Radon measures $\mu^\ell$ and $\mu^\rr$. It is possible to prove that $\supp \mu^\ell\subset \Gamma^c\setminus C_{\xi_\gamma}$ and  $\supp \mu^\rr\subset \Gamma^c\setminus C_{\xi_\gamma}$ and, by the Radon-Nikodym decomposition, there are $f,g\in L^1(\Gamma^c\setminus C_{\xi_\gamma})$ such that 
	\begin{align*}
		\mu^\ell=f\mathcal{H}^1\llcorner (\Gamma^c\cap R) + (\mu^s)^\ell\quad\text{and}\quad
		\mu^\rr=	g\mathcal{H}^1\llcorner (\Gamma^c\cap R) + (\mu^s)^\rr,
	\end{align*}
	with $	(\mu^\ell)^s$ and $	(\mu^r)^s$ are singular measures with respect to $f\ho\llcorner(\Gamma^c\setminus C_{\xi_\gamma})$ and $g\ho\llcorner(\Gamma^c\setminus C_{\xi_\gamma})$ respectively. It is a straight computation to prove the fact that $\mu=\mu^\ell+\mu^\rr$, which implies \eqref{f+g}.\\
	We now focus on the first term on the right-hand side of \eqref{CCstep}.\\
	Notice that since $w_\eps\to \chi_{\O\cap R}$ in $L^1(Q^+\cap R)$, we can apply Theorem $\ref{cctheorem}$ (see Remark \ref{liminfsolution} with $\Sigma=Q$ and $B=R^\ell$) and we get 	
	\begin{align*}
		\liminf_{\eps\to0}\frac{1}{\sigma}\int_{Q^+\cap R^\ell} \Big[\eps|\nabla w_\eps|^2+\frac{1}{\eps}P\big(w_\eps\big)\Big]\psi(u_\eps)\ \de\textbf{x}\geq \int_{\partial^\ast R^\ell}\widetilde{\psi}(f)\ \de\ho+\theta (\mu^s)^\ell(\Gamma^c\setminus C_{\xi_\gamma}).
	\end{align*}
	By applying the same argument to the second term on the right-hand side of \eqref{CCstep}, we get
	\begin{align}\label{surfaceterminsideR}
		\nonumber\liminf_{\eps\to0}\frac{1}{\sigma}\int_{Q^+\cap R} \Big[\eps&|\nabla w_\eps|^2+\frac{1}{\eps}P\big(w_\eps\big)\Big]\psi(u_\eps)\ \de\textbf{x}\\[5pt]
		\nonumber&\geq \int_{\partial^\ast R^\ell}\widetilde{\psi}(f)\ \de\ho+\theta (\mu^s)^\ell(\Gamma^c\setminus C_{\xi_\gamma})\\[5pt]
		\nonumber&\hspace{1cm}+\int_{\partial^\ast R^\rr}\widetilde{\psi}(g)\ \de\ho+\theta (\mu^s)^\rr(\Gamma^c\setminus C_{\xi_\gamma})\\[5pt]
		&\geq \int_{\Gamma^c\setminus C_{\xi_\gamma}} \psi^c(u)\ \de\ho+\theta\mu^s(\widetilde{\Gamma}\setminus C_{\xi_\gamma}),
	\end{align}
	where in the last step we used Definition \ref{psic} and \eqref{f+g}. \\
	
	\emph{Step 2.2 Regular part $\widetilde{\Gamma}$}. We deal with the second term on the right-hand side of \eqref{surfaceenergyR}.  We can see the remaining cut part $C_{\xi_\gamma}$ in $Q^+\setminus R$ as a singular measure with respect to $\ho$, namely 
	\begin{align*}
		u_\eps\cL^2(Q^+\setminus R)\stackrel{\ast}{\wto}\mu\llcorner(Q^+\setminus R)= u\ho\llcorner \widetilde{\Gamma}+\mu^s(Q^+\setminus R)+u\ho\llcorner C_{\xi_\gamma},
	\end{align*}
	where we used \eqref{cxi}.
	Now, we can apply Theorem \ref{cctheorem} and since $\widetilde{\Gamma}\cap C_{\xi_\gamma}\subset Q^+\setminus R$, we obtain
	\begin{align}\label{surfaceenergyoutsideR}
		\nonumber\liminf_{\eps\to0}\frac{1}{\sigma}\int_{Q^+\setminus R} \Big[\eps&|\nabla w_\eps|^2+\frac{1}{\eps}P\big(w_\eps\big)\Big]\psi(u_\eps)\ \de\textbf{x}\\[5pt]\geq &\int_{\widetilde{\Gamma}\cup C_{\xi_\gamma}} \widetilde{\psi} (u)\dho
+ \theta\Big(\mu^s(\widetilde{\Gamma}\cup C_{\xi_\gamma})+\int_{C_{\xi_\gamma}}u\dho\Big).	
	\end{align}
	\emph{Step 3. Conclusion.} Using \eqref{bulktermiliminf}, \eqref{surfaceterminsideR} and \eqref{surfaceenergyoutsideR} we get
	\begin{align}\label{finalstepliminf}
		\nonumber\int_K&W\big(E(v)-E_0(y)\big)\ \de \textbf{x}+\int_{\Gamma^c\setminus C_{\xi_\gamma}} \psi^c(u)\ \de\ho+\theta\mu^s(\Gamma^c\setminus C_{\xi_\gamma})\\[5pt]
		\nonumber&\hspace{0.5cm}+\int_{\widetilde{\Gamma}\cup C_{\xi_\gamma}} \widetilde{\psi} (u)\dho
		+ \theta\Big(\mu^s(\widetilde{\Gamma}\cup C_{\xi_\gamma})+\int_{C_{\xi_\gamma}}u\dho\Big)\\[5pt]
		\nonumber&\leq\liminf_{\eps\to0}\int_{Q^+} (w_\eps+\eta_\eps)W\big(E(v_\eps)-E_0(y)\big)\ \de \textbf{x}\\[5pt]
		\nonumber& \hspace{0.5cm}+\liminf_{\eps\to0}\frac{1}{\sigma}\int_{Q^+} \Big[\eps|\nabla w_\eps|^2+\frac{1}{\eps}P\big(w_\eps\big)\Big]\psi(u_\eps)\ \de\textbf{x}\\[5pt]
		&\leq\liminf_{\eps\to0} \cF_\eps(w_\eps,v_\eps,u_\eps).
	\end{align}
	Now, by letting $\xi_\gamma\to0$, by \eqref{cxi} and since $C_{\xi_\gamma}\to\emptyset$, we get 
	\begin{align}\label{cxito0}
		\mu^s(\Gamma^c\setminus C_{\xi_\gamma})\to\mu^s(\Gamma^c),\quad\mu^s(\widetilde{\Gamma}\cup C_{\xi_\gamma})\to \mu^s(\widetilde{\Gamma})\quad
		\text{and}\quad\int_{C_{\xi_\gamma}}u\ \dho\to0.
	\end{align}
	In conclusion, from \eqref{finalstepliminf} and \eqref{cxito0}, we get the desired liminf inequality
	\begin{align*}
		\cF(\O,v,u)\leq \liminf_{\eps\to0} \cF_\eps(w_\eps,v_\eps,u_\eps).
	\end{align*}
\end{proof}
\section{Limsup inequality}
The proof of the limsup inequality is not a direct application of the Theorem contained in \cite{CarCri}. This is due to the interplay of the phase-field variable in the bulk and in the surface term. For that reason our approach is partially inspired by \cite{BonCha02}.\\
The main result of this section is the mass-constrained limsup inequality.
\begin{theorem}\label{limsup}
	Let $(\O,v,\mu)\in \cA(m,M)$. Then, there exists $(w_\eps,v_\eps,u_\eps)_\eps\subset \cA_p(m,M)$, such that $(w_\eps,v_\eps,u_\eps)\to(\O,v,\mu)$ as $\eps\to0$, and
	\begin{align*}
	\limsup_{\eps\to0} \cF_\eps(w_\eps,v_\eps,u_\eps)\leq	\cF(\O,v,\mu).
	\end{align*}
\end{theorem}
Due to the work carried out in \cite{CriFis}, we can reduce to prove the existence of a recovery sequence for a simpler case. 
We now see, step by step, the argument that allows us to reduce to the case in Theorem \ref{realtheoremlimsup}. 
\begin{remark}
	In what follows, a diagonal argument will be used to obtain the desired result. We note that this is possible as we are in the setting in which the $\Gamma$-convergence can be rephrased in terms of the liminf and limsup inequality.
\end{remark}
We start with an approximation results proved in \cite{CriFis}.
\begin{theorem}\label{limsuptheorem}
 For every $(\O,v,\mu)\in\cA(m,M)$, there exists a sequence $(\O_\eps,v_\eps,\mu_\eps)_\eps\subset \cA_r(m,M)$, with $\mu_\eps=u_\eps\ho\llcorner\Gamma_\eps$, where $u_\eps$ is grid-constant, such that $(\O_\eps,v_\eps,\mu_\eps)\to (\O,v,\mu)$ as $\eps\to0$. Moreover,
\begin{align*}
\limsup_{\eps\to0} \cF(\O_\eps,v_\eps,\mu_\eps)\leq \cF (\O,v,\mu).
\end{align*}
\begin{proof}
The proof is given by \cite[Proposition 7.6 and Proposition 7.7]{CriFis}.
\end{proof}
\end{theorem}
Therefore, Theorem \ref{limsup} will be a consequence of the following result together with a diagonal argument.
\begin{theorem}
	Let $(\O,v,\mu)\in \cA_r(m,M)$ be such that $\mu=u\ho\llcorner\Gamma$ with $u\in L^1(\Gamma)$ grid constant. Then, there exists $(w_\eps,v_\eps,u_\eps)_\eps\subset A_p(m,M)$, such that $(w_\eps,v_\eps,u_\eps)\to(\O,v,\mu)$ as $\eps\to0$, and
	\begin{align*}
		\limsup_{\eps\to0} \cF_\eps(w_\eps,v_\eps,u_\eps)\leq	\int_{\O}W\big(E(v)-E_0(y)\big)\ \de \textbf{x}+\int_{\Gamma}\widetilde{\psi}(u)\dho.
	\end{align*}
\end{theorem}
Using the following approximation result proved in \cite{CriFis} we can simplify even further our statement by considering the functional with $\psi$ instead of $\widetilde{\psi}$ in the surface term.
\begin{theorem}\label{wrigglingthm}
	Let $(\O,v,\mu)\in \cA_r(m,M)$. Then, there exists a sequence $(\O_\eps,v_\eps,\mu_\eps)_\eps\subset \cA_r(m,M)$ such that  $(\O_\eps,v_\eps,\mu_\eps)\to(\O,v,\mu)$ and
	\begin{align*}
		\lim_{\eps\to\infty} \cH(\O_\eps,v_\eps,\mu_\eps)= \int_{\O}W\big(E(v)-E_0(y)\big)\ \de \textbf{x}+\int_{\Gamma}\widetilde{\psi}(u)\dho,
	\end{align*}
	where $\cH$ is defined \eqref{hfunctional}.
\end{theorem}
\begin{proof}
	The proof is given in \cite[Proposition 7.8]{CriFis}.
\end{proof}
\begin{remark}
	From a standard result (see, for instance, \cite[Theorem 5.32 and Remark 5.33]{FL_book}) we can assume in without loss of generality that $\psi$ is convex. 
\end{remark}
Thus, the following Theorem, together with a diagonalisation argument, will give the limsup inequality. 
\begin{theorem}\label{realtheoremlimsup}
For every configuration $(\O,v,\mu)\in \cA_r(m,M)$, there is a sequence $(w_\eps,v_\eps,u_\eps)_\eps\subset \cA_p(m,M)$ such that $(w_\eps,v_\eps,u_\eps)\to(\O,v,\mu)$ as $\eps\to0$. Moreover,
\begin{align*}
	\limsup_{\eps\to0} \cF_{\eps}(w_\eps,v_\eps,u_\eps)\leq \cH (\O,v,\mu).
\end{align*}
\end{theorem}
\begin{proof}
\emph{Step 1. $C^\infty$ approximation.} Let $(\O,v,\mu)\in \cA_r(m,M)$. We approximate $h$ by convolution in order to get a $\cC^\infty$ approximant sequence $(g_\eps)_\eps$ such that the mass constraint is satisfied, for every $\eps>0$.\\
Consider a convolution kernel $\rho\in\cC_c^\infty(\R)$ (namely, $\int_\R\rho\ \de x=1$, $\rho\geq 0$ and $\supp \rho\subset [-1/2,1/2]$). Let, for each $x\in\R$,
 \begin{align*}
 	\rho_\eps(x)\coloneqq\frac{1}{\eps}\rho\Big(\frac{x}{\eps}\Big)
 \end{align*}
 and $g_\eps:(a,b)\to\R$ be defined as
	\begin{align*}
		g_\eps(x)\coloneqq h\ast \rho_\eps(x).
	\end{align*}
Let $\O_\eps$ be the sub-graph of $g_\eps$. By standard results, $\Gamma_\eps\coloneqq\partial \O_\eps\cap\big((a,b)\times\R\big)$ is a $\cC^\infty$ curve and since $g_\eps\to h$ uniformly as $\eps\to0$ we have
	\begin{align}\label{convolutionhausdconv}
	\R^2\setminus \O_\eps\stackrel{H}{\to} \R^2\setminus\O.
	\end{align}
	 In addition,
	\begin{align}\label{convergencelength}
		\lim_{\eps\to0}\ho(\Gamma_\eps)=\ho(\Gamma),
	\end{align}  
	Conisider a grid $(R^j)_{j=1}^N$ for which $\mu$ is grid-constant. From $(iii)$ of Definition \ref{deltacoverdef}, the uniform convergence and \eqref{convergencelength}, we can assume that there exists $\eps_0>0$ such that for every $\eps<\eps_0$ and for each $j=1,\dots,N$  we have
	\begin{align*}
		\Gamma\cap R^j\neq\emptyset\quad\text{and}\quad\Gamma_\eps\cap R^j\neq\emptyset.
	\end{align*} 
	We also note that from the previous assumption that
		\begin{align}\label{convergencelength2}
		\lim_{\eps\to0}\ho(\Gamma_\eps\cap R^j)=\ho(\Gamma\cap R^j),
	\end{align}  
	for each $j=1,\dots,N$.\\
We define $\overline{u}_\eps:\Gamma_\eps\to\R$ as
	\begin{align}\label{sameintersection}
\overline{u}_{\eps|\Gamma_\eps\cap R^j}\coloneqq \frac{\ho(\Gamma\cap R^j)}{\ho(\Gamma_\eps\cap R^j)}u^j.
	\end{align}
By assumption on $u$ we get that
	\begin{align*}
		\int_{\Gamma_\eps} \overline{u}_\eps\ \dho=\sum_{j=1}^N \overline{u}_\eps\ho (\Gamma_\eps\cap R^j)=\sum_{j=1}^N u^j\ho (\Gamma\cap R^j)=m.
	\end{align*}
	By using standard properties of the convolution, we obtain
	\begin{align*}
	\int_a^b g_\eps\ \de x=\int_a^b h\ast \rho_\eps\ \de x=\Big(\int_a^b h(x)\ \de x\Big)\Big(\int_\R \rho_\eps(x)\ \de x\Big)=M.
	\end{align*}
We are left to define a displacement sequence $\overline{v}_\eps \in H^1(\O_\eps;\R^2)$. We claim that for every $x\in(a,b)$ it holds
\begin{align}\label{gepsestimatefordisplacement}
	g_\eps(x)\leq h(x) + \eps L,
	\end{align}
where $\ell$ is the Lipschitz constant of $h$. We have
\begin{align*}
	g_\eps(x)=h\ast\rho_\eps(x)&=\int_\R h(t)\rho_\eps(x-t)\ \de t \\[5pt]&\leq \int_a^b h(x)\rho_\eps(x-t)\ \de t+\int_a^b \abs{h(t)-h(x)}\rho_\eps(x-t)\ \de t\\[5pt]
	&\leq h(x) +\ell\eps.
\end{align*}
Thanks to \eqref{gepsestimatefordisplacement}, we have
\begin{align*}
\{(x,y-\eps \ell):(x,y)\in\O_\eps\}\subset\subset\O
\end{align*}
and thus, the function $v_\eps: Q\to \R^2$, set as
\begin{align*}
	v_\eps(x,y)\coloneqq v (x,y-\eps \ell),
\end{align*}
 is well defined in a neighbourhood of $\O_\eps$. \\
We conclude that $(\O_\eps,v_\eps,\mu_\eps)_\eps\subset\cA_r(m,M)$, where $\mu_\eps\coloneqq u_\eps\ho\llcorner\Gamma_\eps$, and that $(\O_\eps,v_\eps,u_\eps)\to(\O,v,\mu)$ from \eqref{convolutionhausdconv} and the fact that $v_\eps$ is a translation of  $v$ and $\abs{\abs{v_\eps-v}}_{H^1(Q;\R^2)}\to0$ (see \cite[Theorem 4.26]{Brezis})  as $\eps\to0$.\\
Now we prove that we also have the approximation of the energy.
First, we have
\begin{align}\label{W1}
\nonumber	\Big|\int_{\O_\eps}W\big(E(v_\eps)-E_0(y)\big)\ \de \textbf{x} &- \int_\O W\big(E(v)-E_0(y)\big)\ \de \textbf{x}\Big|\\[5pt]
	\nonumber&\leq\int_{\O_\eps}\abs{W\big(E(v_\eps)-E_0(y)\big)- W\big(E(v)-E_0(y)\big)}\ \de \textbf{x}\\[5pt]
	&\hspace{0.5cm}+\int_{\O_\eps\setminus\O}W\big(E(v)-E_0(y)\big)\ \de \textbf{x}.
\end{align}
The second term on the right-hand side of \eqref{W1} goes to zero by Lebesgue Dominated Convergence Theorem. Regarding the first term, by a continuity argument for $W$ (as it is a quadratic form), we obtain that for every $\lambda>0$ there is a $\overline{\eps}>0$ such that for every $\eps<\overline{\eps}$ we have
\begin{align}\label{W2}
	\abs{W\big(E(v_\eps)-E_0(y)\big)-W\big(E(v)-E_0(y)\big)}\leq\lambda.
\end{align}
From \eqref{W1}
 and by taking into account \eqref{W2}, we obtain
 \begin{align}\label{W3}
\nonumber\Big|\int_{\O_\eps}W\big(E(v_\eps)-E_0(y)\big)\ \de \textbf{x} &- \int_\O W\big(E(v)-E_0(y)\big)\ \de \textbf{x}\Big|\\[5pt]
	&\leq \lambda \abs{\O_\eps}+\int_{\O_\eps\setminus\O}W\big(E(v)-E_0(y)\big)\ \de \textbf{x}\to0,
 \end{align} 
 as $\eps\to0$.
 Moreover, by \eqref{convergencelength} and \eqref{sameintersection}, we have
\begin{align}\label{W4}
\int_{\Gamma_\eps}\psi(\overline{u}_\eps)\ \de\ho=	\sum_{j=1}^N\psi\Big(\frac{\ho(\Gamma\cap R^j)}{\ho(\Gamma_\eps\cap R^j)}u^j\Big)\ho(\Gamma_\eps\cap R^j)\to\int_\Gamma \psi(u)\ \dho,
\end{align}
as $\eps\to0$, where we used the convexity of $\psi$ and \eqref{convergencelength2}. From $\eqref{W3}$ and $\eqref{W4}$ we can conclude that
\begin{align*}
	\lim_{\eps\to0} \cH(\O_\eps,v_\eps,u_\eps)=\cH(\O,v,\mu).
\end{align*}
Therefore, in the sequel we can assume $h$ to be of class $\cC^\infty$.\\

\emph{Step 2. Displacement and phase-field sequences.} Given $\eps>0$, consider the almost optimal profile problem
\begin{equation}
	\begin{cases}
		\eps^2	\abs{\gamma_\eps'(t)}^2=P(\gamma_\eps(t))+\sqrt{\eps}\quad 0\leq t\leq1,\\[5pt]
		\gamma_\eps(0)=1,\\[5pt]
		\gamma_\eps(1)=0.
	\end{cases}
	\label{optimalprofile2}
\end{equation}
We have that \eqref{optimalprofile2} has an unique solution $\gamma_\eps\in\cC^1([0,1])$ and $0\leq\gamma_\eps\leq1$. Moreover, we can extend $\gamma_\eps$ on $\R$ by setting $\gamma_\eps=1$ for $x<0$ and $\gamma_\eps=0$ for $x>1$. Note that, for every $\eps>0$,
\begin{align}\label{gammaepsbound}
\abs{\gamma_\eps'(t)}=\frac{1}{\eps}\sqrt{P(\gamma_\eps(t))+\sqrt{\eps}}<\frac{C}{\eps},
\end{align}
as the potential $P$ is bounded in $[0,1]$.\\
We define the phase-field approximating sequence $z_\eps :Q^+ \to \R$ as
\begin{align*}
	z_\eps(\bfx)\coloneqq \gamma_\eps\Big(\frac{d_{\O}(\bfx)}{\eps}\Big).
\end{align*}
We have that $z_\eps\in H^1(Q^+)$, indeed from the regularity of $\O$ follows that $d_\O\in\cC^2(\R^2)$, and therefore $\gamma_\eps\in\cC^1(\R)$. 
We first prove that $z_\eps\to \chi_\O$ in $L^1_\loc(Q)$. Take any compact $K\subset Q$. We have that
\begin{align}\label{zepstochi}
		\nonumber\int_{Q^+ \cap K} \abs{z_\eps(\bfx)-\chi_\O(\bfx)}\ \de\bfx &=\int_{\{\bfx\in Q^+\cap K:\  0\leq\abs{d_\O(\bfx)}\leq\eps\}} \abs{z_\eps(\bfx)-\chi_\Omega(\bfx)}\ \de \bfx\\[5pt]
		&\leq |\{\bfx\in Q^+\cap K:\  0\leq\abs{d_\O(\bfx)}\leq\eps\}|.
\end{align}
Using standard properties of the Minkowski content (see \cite[Definition 2.100 and Theorem 2.104]{AFP}) we have that
\begin{align}\label{minkowski}
	\lim_{\eps\to0}\frac{1}{2\eps}|\{\bfx\in Q^+\cap K:\  0\leq\abs{d_\O(\bfx)}\leq\eps\}|= \ho(\partial\O\cap K),
\end{align}
as $\eps\to0$.
Therefore, from \eqref{zepstochi} and \eqref{minkowski}, we can conclude that for every compact set $K$, 
\begin{align*}
\lim_{\eps\to0}\int_{Q^+\cap K} \abs{z_\eps(\bfx)-\chi_\O(\bfx)}\ \de\bfx=0.
\end{align*}
In order to get the mass constraint satisfied, we set
\begin{align*}
	\alpha_\eps\coloneqq \frac{M}{\int_{Q^+} z_\eps(\bfx)\ \de\bfx}.
\end{align*}
Note that, from similar computation to \eqref{zepstochi}, we can deduce
\begin{align}\label{intzeps}
	\int_{Q^+} z_\eps(\bfx)\ \de\bfx \geq M.
\end{align}
Now, since
\begin{align*}
	\int_{Q^+} |z_\eps|\ \de\bfx\to M,
\end{align*}
and from \eqref{intzeps}, we have \begin{align}\label{alphaeps}
	\alpha_\eps\leq1\quad \text{and}\quad\lim_{\eps\to0}\alpha_\eps=1.
\end{align}
We define a rescaled phase-field variable $w_\eps:\R^2\to\R$ as
\begin{align*}
	w_\eps(x,y)\coloneqq z_\eps(x,\alpha_\eps y),
\end{align*}
In such a way, for each $\eps>0$, $w_\eps$ satisfies 
\begin{align*}
	\int_{Q^+} w_\eps(\bfx)\ \de \bfx=M,
\end{align*}
and $w_\eps\to \chi_{\O}$ in $L^1_\loc(Q)$, as $\eps\to0$. 
Moreover, for every $\eps>0$, $w_\eps\in H^1(Q,\R)$. Indeed, from \eqref{zepstochi},  \eqref{alphaeps} and up to a change of variable, it is enough to show that, for every $\eps\geq0$,  $z_\eps\in H^1(Q,\R)$. We have that, by using \eqref{minkowski}, we have $z_\eps\in L^2(Q^+)$.
For the gradient, we have the following estimate,
\begin{align*}
\int_{Q^+} |\nabla z_\eps|^2\ \de\bfx = \int_{\{\bfx\in Q^+:\  0\leq\abs{d_\O(\bfx)}\leq\eps\}}\Big|\frac{\gamma'_\eps(d_\O(\bfx))}{\eps}\Big|^2\ \de\bfx \leq \frac{C^2}{\eps}\ho(\partial \O)<\infty,
\end{align*}
	where $C>0$ is as in \eqref{gammaepsbound}. Thus, we obtain that $z_\eps\in H^1(Q,\R)$, for every $\eps>0$, which implies $w_\eps\in H^1(Q,\R)$. \\
	Let $\ell$ be the Lipschitz constant of $h$ and assume $\ell\geq1$. We claim that if $(x,y)\in Q\setminus\O$ is such that $d_\O(x,y)<\eps$, then
\begin{align}\label{yelleps}
	(x,y-\ell\eps)\in\O.
\end{align}
Indeed, take any point $(\bar{x},h(\bar{x}))\in\Gamma$. We have
\begin{align*}
	\abs{h(\bar{x})-h(x)}\leq \ell \abs{\bar{x}-x},
\end{align*}
from which we can infer, by adding and subtracting $y$ and by using $\ell\geq1$, that
\begin{align*}
	y-h(x)\leq \ell(\abs{\bar{x}-x}+\abs{h(\bar{x})-y}).
\end{align*}
Now, if we take the infimum on both sides for $\bar{x}\in\Gamma$, we obtain
\begin{align*}
	y-h(x)\leq\ell d_\O(x,y)\leq \ell\eps,
\end{align*}
for which \eqref{yelleps} follows. The case in which $\ell<1$ is easier and therefore omitted. \\
We define the approximating displacement sequence $v_\eps:Q\to\R^2$ as
\begin{align*}
	v_\eps(x,y)\coloneqq\begin{cases}
	v(x,y-\ell\eps) w_\eps(x,y-\ell\eps)\quad &\text{if } w_\eps(x,y-\ell\eps)>0,\\[5pt]
		0\quad&\text{else},
	\end{cases}
\end{align*}
Since $h\in\cC^\infty(a,b)$, the definition of $v_\eps$ is well posed and $v_\eps\in H^1(Q;\R^2)$. Indeed, for every compact set $K\subset Q$ 
\begin{align}\label{vepsl2}
	\int_{Q\cap K} |v_\eps|^2\ \de\bfx
	\leq \abs{\abs{v}}^2_{L^2(Q;\R^2)}
\end{align}
and therefore $v_\eps \in L^2_\loc (Q;\R^2)$. Now, we set $\widetilde{w}_\eps(x,y)\coloneqq w_\eps(x,y-\ell\eps)$ and estimate
\begin{align}\label{vepsh1}
	\nonumber\int_{Q\cap K} |\nabla v_\eps|^2\ \de\bfx&\leq \int_{Q\cap K} \Big(\abs{\nabla v_\eps(x,y) \widetilde{w}_\eps(x,y)}^2+\abs{v_\eps(x,y)\otimes \nabla \widetilde{w}_\eps(x,y)}^2\\[5pt]&\hspace{1cm}+2\abs{\nabla v_\eps(x,y) \widetilde{w}_\eps(x,y)}\abs{v_\eps(x,y)\otimes \nabla \widetilde{w}_\eps(x,y)}\Big)\ \de \bfx.
\end{align}
 To see that $v_\eps\in H^1_\loc(Q;\R^2)$,  taking into account \eqref{vepsl2}, we prove that the right-hand side of \eqref{vepsh1} is bounded. The first term is easily estimated by the Sobolev norm of the gradient of $v$, indeed, set $Q\cap K_\eps\coloneqq Q\cap K+(0,\ell\eps)$, we have
 \begin{align*}
 \int_{Q\cap K}\abs{\nabla v_\eps(x,y)\widetilde{w}_\eps(x,y)}^2 \de\bfx&\leq  \int_{Q\cap K}\abs{\nabla v_\eps(x,y)}^2 \de\bfx\\[5pt]
 &=\int_{Q\cap K_\eps} \abs{\nabla v}^2\ \de\bfx\\[5pt]
 &\leq \abs{\abs{\nabla v}}_{L^2(Q)}.
 \end{align*}

 We estimate the second one and then conclude for the entire right-hand side of \eqref{vepsh1}. We have that
\begin{align*}
\int_{Q\cap K} \abs{v_\eps(x,y)\otimes \nabla \widetilde{w}_\eps(x,y)}^2\ \de \bfx&\leq	\int_{Q\cap K} \abs{v_\eps(x,y)}^2\abs{\nabla \widetilde{w}_\eps(x,y)}^2\ \de\bfx\\[5pt]
 &\leq \int_{Q\cap K_\eps}\abs{v(x,y)}^2\abs{\frac{\gamma'_\eps(d_\O(x,\alpha_\eps y))}{\eps^2}}^2 \de\bfx\\[5pt]
 &= \frac{C^2\abs{\abs{v}}^2_{L^2(Q;\R^2)}}{\eps^4},
\end{align*}
where we used \eqref{gammaepsbound} and \eqref{minkowski}. Therefore, for each $\eps$, $v_\eps\in H^1_\loc(Q;\R^2)$. \\
By using similar computations, it is possible to show that $v_\eps\to v$ in $L^2_\loc(Q;\R^2)$. Consider
\begin{align*}
	E_\eps (y)\coloneqq E_0(y-\ell\eps)=\begin{cases}
			t e_1 \otimes e_1\quad&\text{if}\ y\geq\ell\eps ,\\[5pt]
			0\quad&\text{if}\ y<\ell\eps.
		\end{cases}
\end{align*}
It follows, by definition, that 
\begin{align*}
	E(v_\eps(\bfx))=\begin{cases}
		E\big(v_\eps(x,y)\big)\widetilde{w}_\eps(x,y)+v_\eps(x,y)\odot \nabla \widetilde{w}_\eps(x,y)\quad &\text{if } \widetilde{w}_\eps(x,y)>0,\\[5pt]
		0\quad&\text{else}.
	\end{cases}
\end{align*}
\begin{figure}
	\begin{center}
		\includegraphics[scale=0.4]{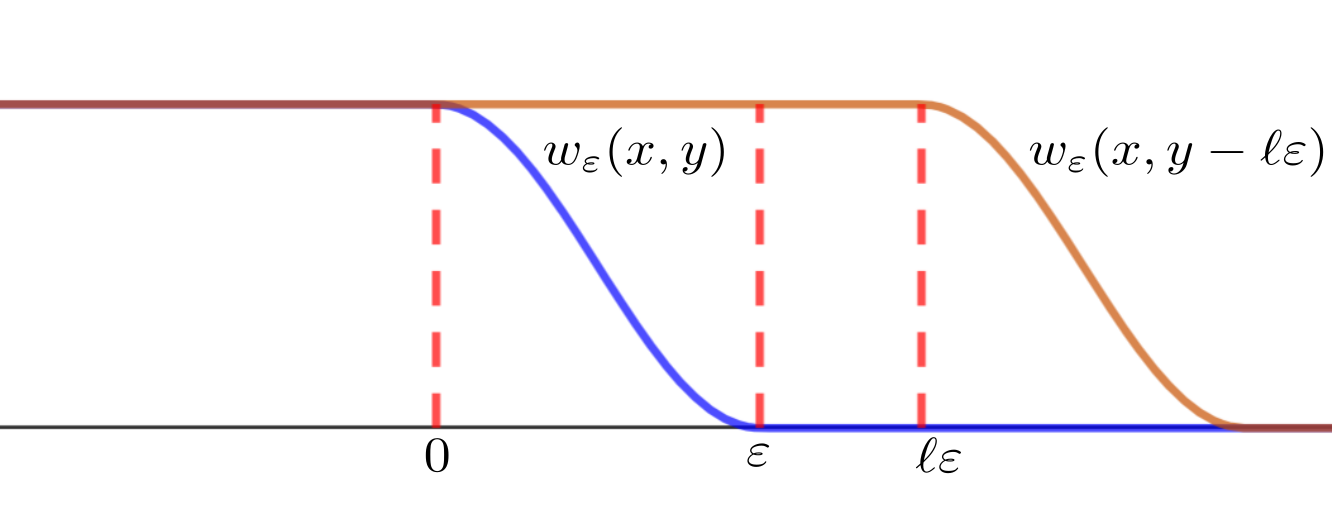}
		\caption{For each $x$, the function $\widetilde{w}_\eps(x,y)=w_\eps(x,y-\ell y)$ allows us to se the potential term $W$, where $\nabla \widetilde{w}_\eps=0$.}
		\label{figweps}
	\end{center}
\end{figure}
The reason for which the following computation is made with the translation of $w_\eps$, namely $\widetilde{w}_\eps$, is in order to avoid to see the gradient of $w_\eps$ in the area where it equals $1$, as it is made clear in Figure \ref{figweps}.\\
First, we set
\begin{align*}
	\widetilde{A}_\eps&\coloneqq\left\{\widetilde{w}_\eps(x,y)=1,\ h(x)+\eps<y<h(x)+\eps(\ell-1)\right\},\\[5pt]
	A_\eps&\coloneqq\left\{w_\eps(x,y)=1,\ h(x)+\eps(1-\ell)<y<h(x)+-\eps\right\}.
\end{align*}
We compute
\begin{align}\label{finallimsup1}
\nonumber\int_{Q^+}\big(w_\eps(\bfx)+\eta_\eps\big) W\big(&E(v_\eps)-E_\eps (y)\big)\ \de\bfx
\nonumber\\[5pt]\nonumber&=(1+\eta_\eps)\Big[\int_{\{w_\eps>0\}}  W\big(E(v_\eps)-E_\eps (y)\big)\ \de\bfx\\[5pt]
\nonumber&\hspace{0.5cm}+\int_{\widetilde{A}_\eps} W\big(E(v_\eps)-E_\eps (y)\big)\ \de\bfx\Big]\\[5pt]
&\hspace{0.5cm}+\eta_\eps\int_{\{0<\widetilde{w}_\eps<1\}} W\big(E(v_\eps)-E_\eps (y)\big)\ \de\bfx.
\end{align}
For the first term on the right-hand side of \eqref{finallimsup1}, we have
\begin{align}\label{finallimsup1.5}
\nonumber	\int_{\{w_\eps>0\}}  W\big(E(v_\eps)-E_\eps (y)\big)\ \de\bfx&=\int_{\{w_\eps(x,y+\ell\eps)>0\}}  W\big(E(v)-E_0(y)\big)\ \de\bfx\\[5pt]
	&\leq\int_\O W\big(E(v)-E_0(y)\big)\ \de\bfx,
\end{align}
since $\{w_\eps(x,y+\ell\eps)>0\} \subset \O$. Proceeding with the second term, we have
\begin{align}\label{finallimsup2}
\int_{\widetilde{A}_\eps} W\big(E(v_\eps)-E_\eps (y)\big)\ \de\bfx=\int_{A_\eps} W\big(E(v)-E_0(y)\big)\ \de\bfx,
\end{align}
which goes to $0$, as $\eps\to0$, by Dominated Convergence Theorem. Finally
\begin{align}\label{finallimsup3}
\nonumber\eta_\eps\int_{\{0<\widetilde{w}_\eps<1\}} &W\big(E(v_\eps)-E_\eps (y)\big)\ \de\bfx\\[5pt]
\nonumber&\leq K\eta_\eps\int_{\{0<\widetilde{w}_\eps<1\}}	\abs{E\big(v_\eps)\widetilde{w}_\eps-E_\eps (y)}^2+\abs{v_\eps(x,y)}^2 \abs{\nabla \widetilde{w}_\eps(x,y)}^2\\[5pt]
&\hspace{1cm}+2\abs{E\big(v_\eps)\widetilde{w}_\eps-E_\eps (y)}\abs{v_\eps(x,y)} \abs{\nabla \widetilde{w}_\eps(x,y)}\ \de\bfx,
\end{align}
where $K>0$ is the constant from the upper bound of the growth of $W$.\\
We estimate the three terms in the sum on the right-hand side of \eqref{finallimsup3}. We have
\begin{align}\label{finallimsup4}
\nonumber\eta_\eps\int_{\{0<\widetilde{w}_\eps<1\}} 	\abs{E\big(v_\eps)\widetilde{w}_\eps-E_\eps(y)}^2\ \de\bfx&=\eta_\eps\int_{\{0<w_\eps<1\}}\abs{E(v)w_\eps-E_0(y)}^2\ \de\bfx\\[5pt]
\nonumber&\leq\eta_\eps\int_{\{0<w_\eps<1\}}\abs{E(v)w_\eps}^2+\abs{E_0(y)}^2+2\abs{E(v)w_\eps}\abs{E_0(y)}\ \de\bfx\\[5pt]
\nonumber&\leq\eta_\eps\int_{\{0<w_\eps<1\}}|E(v)|^2\ \de\bfx+\int_{\{0<w_\eps<1\}}\abs{E_0(y)}^2\ \de\bfx\\[5pt]
&\hspace{0.5cm}+2\int_{\{0<w_\eps<1\}}\abs{E(v)}\abs{E_0(y)}\ \de\bfx.
\end{align}
Now, all the terms on the right-hand of \eqref{finallimsup4} are bounded, since $\nabla v\in L^2(\{0<w_\eps<1\};\R^2)\subset L^1(\{0<w_\eps<1\};\R^2)$ and since $\abs{E_0(y)}$ is constant on $\{0<w_\eps<1\}$. From that, we conclude that
\begin{align}\label{finallimsup5}
	\eta_\eps\int_{\{0<\widetilde{w}_\eps<1\}} 	\abs{E\big(v_\eps)\widetilde{w}_\eps-E_\eps(y)}^2\ \de\bfx\to0,
\end{align} 
since $\eta_\eps\to0$, as $\eps\to0$.\\
For the second term on the right-hand side of \eqref{finallimsup3}, by taking into account \eqref{gammaepsbound}, we have
\begin{align}\label{finallimsup6}
\nonumber\eta_\eps\int_{\{0<\widetilde{w}_\eps<1\}}  \abs{v_\eps}^2 \abs{\nabla \widetilde{w}_\eps}^2\ \de\bfx
&\leq C^2\frac{\eta_\eps}{\eps^2} \int_{\{0<w_\eps<1\}} \abs{v}^2 \de\bfx \\[5pt] &\leq C^2\frac{\eta_\eps}{\eps^2}\abs{\abs{v}}_2 \abs{\{\bfx\in Q^+:\  0<d_\O(\bfx)\leq\eps\}}\to0,
\end{align}
as $\eps\to0$, indeed, by \eqref{minkowski} \begin{align*}
	\frac{1}{\eps}\abs{\{\bfx\in Q^+:\  0<d_\O(\bfx)\leq\eps\}}\to \cH^1(\Gamma)
\end{align*} 
and $\eta_\eps/\eps\to 0$ by assumption.\\
Finally, the third term on the right-hand side of \eqref{finallimsup3} goes to 0 by a similar approach to the one used for \eqref{finallimsup5} and \eqref{finallimsup6}. \\
In conclusion, from \eqref{finallimsup1}, \eqref{finallimsup1.5}, \eqref{finallimsup5} and \eqref{finallimsup6} we obtain
\begin{align}\label{limsupbulk}
	\limsup_{\eps\to0}\int_{Q^+}\big(w_\eps(\bfx)&+\eta_\eps\big) W\big(E(v_\eps)-E_\eps(y)\big)\ \de\bfx\leq\int_{\Omega}  W\big(E(v)-E_0(y)\big)\ \de\bfx.
\end{align}

	\emph{Step 3. Density sequence.} 
Since $u$ is grid-constant, there exists a family of squares $\{R^j\}_{j=1}^N$, for which
\begin{align*}
	u_{|R^j\cap\Gamma}=u^j\geq0.
\end{align*}
Note that for every $\bfx\in Q$ such that $\abs{d_\O(\bfx)}>\eps$, we have $w_\eps(\bfx)=0$ or $w_\eps(\bfx)=1$. 
Up to further reducing $\eps$, from \cite[Lemma 5.4]{CriFis}, we can assume that, for every $j=1,\dots,N$,
\begin{align*}
	\{\bfx\in\R^2\ : \ \abs{d_\O(\bfx)}<\eps\} \subset R^j,
\end{align*}
so that if $\bfx\notin R^j$, for every $j=1,\dots,N$, we have
\begin{align}\label{pwequalzero}
	P(w_\eps)=|\nabla w_\eps|^2=0.
\end{align}
We set
\begin{align*}
	p_\eps^j\coloneqq\frac{1}{\sigma}\int_{Q^+\cap R^j} \Big(\eps |\nabla w_\varepsilon|^2+\frac{1}{\eps}P(w_\varepsilon) \Big)\ \de \bfx
\end{align*}
On each $R^j$, we define the approximating density $u_\eps:\R^2\to\R$ as
\begin{align*}
	u_{\eps|R^j}(\bfx)\coloneqq  \begin{cases}
		\displaystyle\frac{\mathcal{H}^1(\Gamma\cap R^j)}{ p_\eps^j  }u^j\quad&\bfx\in R^j\\[5pt]
		0\quad&\text{else}
	\end{cases}
\end{align*}
and notice that 
\begin{align}\label{limitpeps}
	\lim_{\eps\to 0}\frac{\mathcal{H}^1(\Gamma\cap R^j)}{ p_\eps^j  }=1,
\end{align}
for every $j=1,\dots,N$.
We note that if we set
\begin{align*}
	\mu_\eps\coloneqq  \frac{1}{\sigma}\Big(\eps |\nabla w_\varepsilon|^2+\frac{1}{\eps}P(w_\varepsilon) \Big)u_\eps \cL^2\llcorner Q^+,
\end{align*}
we have 
\begin{align*}         
	\mu_\eps(\R^2)=\frac{1}{\sigma}\int_{Q^+}\Big(\eps |\nabla w_\varepsilon|^2+\frac{1}{\eps}P(w_\varepsilon) \Big) u_\eps\ \de\bfx=\sum_{j=1}^N u^j\ho(\Gamma\cap R^j)=m.
\end{align*}
Moreover, $\mu_\eps\stackrel{\ast}{\wto}\mu$, as $\eps\to0$, indeed, take any $\varphi\in\cC_c(Q)$, we have
\begin{align*}
	\int_{Q^+}& \varphi\ \de\mu_\eps = \frac{1}{\sigma}\int_{Q^+} \Big(\eps |\nabla w_\varepsilon|^2+\frac{1}{\eps}P(w_\varepsilon) \Big)u_\eps\varphi\ \de\bfx\\[5pt]
	&=\sum_{j=1}^{N} u^j   \Big(\frac{\mathcal{H}^1(\Gamma\cap R^j)}{ p_\eps^j }-1\Big)\int_{Q^+\cap R^j} \Big(\eps |\nabla w_\varepsilon|^2+\frac{1}{\eps}P(w_\varepsilon) \Big)\varphi\ \de\bfx\\[5pt]
	&\hspace{1cm}+\sum_{j=1}^{N}\frac{u^j}{\sigma}\int_{Q^+\cap R^j} \Big(\eps |\nabla w_\varepsilon|^2+\frac{1}{\eps}P(w_\varepsilon) \Big)\varphi\ \de\bfx\\[5pt]
	&\to\sum_{j=1}^{N} u^j \ho(\Gamma\cap R^j),
\end{align*}
as $\eps\to0$, 
where we used \cite[Proposition 5.4,  Lemma 6.2]{CarCri}.\\
Now, we show that
\begin{align*}
\limsup_{\eps\to0} \int_{Q^+}\Big(&\eps|\nabla w_\eps|^2+\frac{1}{\eps}P(w_\eps)\Big)\psi(u_\eps)\ \de\bfx \leq\sigma\int_\Gamma \psi(u)\ \de\ho.
\end{align*}
First we have, by applying a change of variable and taking into account \eqref{alphaeps} and \eqref{pwequalzero}, that
\begin{align}\label{finalsteplimsup1}
\nonumber\int_{Q^+}&\Big(\eps|\nabla w_\eps|^2+\frac{1}{\eps}P(w_\eps)\Big)\psi(u_\eps)\ \de\bfx \\[5pt]
&\leq\sum_{j=1}^{N} \frac{1}{\alpha_\eps}\int_{R^j}\Big(\eps\abs{\nabla z_\eps}^2+\frac{1}{\eps}P(z_\eps)\Big)\psi\Big(\frac{\mathcal{H}^1(\Gamma\cap R^j)}{ p_\eps^j }u^j\Big)\ \de\bfx.
\end{align}
Morover, since $\psi$ is continuous (as it is convex by assumption), for every $\lambda>0$, there is $\eps_0>0$ such that, for every $\eps<\eps_0$, we have that, by taking into account \eqref{limitpeps}, that
\begin{align*}
\abs{	\psi\Big(\frac{\mathcal{H}^1(\Gamma\cap R^j)}{ p_\eps^j }u^j\Big)-\psi(u^j)}<\lambda.
\end{align*}
Therefore, by adding and subtracting $\psi(u^j)$ and for $\eps<\eps_0$, from \eqref{finalsteplimsup1}, we have
\begin{align}\label{finalsteplimsup2}
\nonumber\int_{Q^+}\Big(\eps|\nabla w_\eps|^2&+\frac{1}{\eps}P(w_\eps)\Big)\psi(u_\eps)\ \de\bfx \\[5pt]
\nonumber&\leq\sum_{j=1}^{N} \frac{\lambda}{\alpha_\eps}\int_{R^j}\Big(\eps\abs{\nabla z_\eps}^2+\frac{1}{\eps}P(z_\eps)\Big)\ \de\bfx\\[5pt]
&\hspace{0.5cm}+\sum_{j=1}^{N}\frac{\psi(u^j)}{\alpha_\eps}\int_{R^j}\Big(\eps\abs{\nabla z_\eps}^2+\frac{1}{\eps}P(z_\eps)\Big)\ \de\bfx
\end{align}
Finally, since the sequence $(z_\eps)_\eps$ makes use of the solution of problem \eqref{optimalprofile2}, we can apply the same argument as in \cite[Proposition 2]{Mod87} for the limsup inequality. By taking the limsup on both sides of \eqref{finalsteplimsup2} and then, by letting $\lambda\to0$, we obtain
\begin{align}\label{limsupsurf}
\nonumber
&\limsup_{\eps\to0} \int_{Q^+}\Big(\eps|\nabla w_\eps|^2+\frac{1}{\eps}P(w_\eps)\Big)\psi(u_\eps)\ \de\bfx \\[5pt]
\nonumber&\leq\sigma\sum_{j=1}^N\psi(u^j) \ho(\Gamma\cap R^j) \\[5pt]
&=\sigma\int_\Gamma \psi(u)\ \de\ho.
\end{align}

\emph{Step 4. Conclusion.} Using \eqref{limsupbulk} and \eqref{limsupsurf} we obtain that
\begin{align*}
\limsup_{\eps\to0} \cF_\eps(w_\eps,v_\eps,u_\eps)\leq \cH(\O,v,u),
\end{align*}
which concludes the proof of Theorem \ref{realtheoremlimsup}.
\end{proof}
By putting together Theorems \ref{realtheoremlimsup} and  \ref{liminftheorem} we  proved Theorem \ref{mainresult1} and \ref{mainresult2}.

\bibliographystyle{siam}
\def\url#1{}
\bibliography{bibliography}

\end{document}